\documentclass[a4paper,12pt,leqno]{amsart}
\usepackage[a4paper,margin=2.3cm]{geometry}
\usepackage{amssymb,graphicx}
\newif\ifFIRST\newdimen\MAXright\MAXright0pt
\def\dynkin{\mbox\bgroup\scriptsize\FIRSTtrue\hskip.5em\let\SMALL\relax%
\setbox1\hbox{$\diagup$}%
\setbox2\hbox{$\diagdown$}%
\setbox0\hbox to2\wd1{\hrulefill}%
\setbox8\hbox to.5\wd1{\hrulefill}%
\setbox3\hbox{$\bullet$}%
\setbox4\hbox{$\times$}%
\def\root##1{\ifFIRST\setbox5\hbox{$##1$}\ifdim\wd5>1.3em%
\hskip-.5em\hskip.5\wd5\fi\fi\FIRSTfalse%
\hskip-.25em\raise1.5\wd3\hbox to0pt{\hss\hskip.45em$%
##1$\hss}\copy3\hskip-.25em\setbox6\hbox{$##1$}%
\MAXright\wd6}%
\def\droot##1{\ifFIRST\setbox5\hbox{$##1$}\ifdim\wd5>1.3em%
\hskip-.5em\hskip.5\wd5\fi\fi\FIRSTfalse%
\hskip-.25em\lower1.8\wd3\hbox to0pt{\hss\hskip.45em$%
##1$\hss}\copy3\hskip-.25em\setbox6\hbox{$##1$}%
\MAXright\wd6}%
\def\rroot##1{\hskip-.25em\copy3\hbox to0pt{\hskip.3em$##1$\hss}%
\hskip-.25em\setbox6\hbox{\hskip.6em$##1##1$}%
\MAXright\wd6}%
\def\norroot##1{\hskip-.36em\copy4\hbox to0pt{\hskip.3em$##1$\hss}%
\hskip-.48em\setbox6\hbox{\hskip.6em$##1##1$}%
\MAXright\wd6}%
\def\noroot##1{\ifFIRST\setbox5\hbox{$##1$}\ifdim\wd5>1.3em%
\hskip-.5em\hskip.5\wd5\fi\fi\FIRSTfalse%
\hskip-.36em\raise1.5\wd3\hbox to0pt{\hss\hskip.6em$%
##1$\hss}\copy4\hskip-.38em\setbox6\hbox{$##1$}%
\MAXright\wd6}%
\def\nodroot##1{\ifFIRST\setbox5\hbox{$##1$}\ifdim\wd5>1.3em%
\hskip-.5em\hskip.5\wd5\fi\fi\FIRSTfalse%
\hskip-.36em\lower1.8\wd3\hbox to0pt{\hss\hskip.6em$%
##1$\hss}\copy4\hskip-.38em\setbox6\hbox{$##1$}%
\MAXright\wd6}%
\def\link{\raise.25em\copy0}%
\def\minilink{\raise.25em\copy8}%
\def\llink##1{\raise.35em\copy0\hskip-\wd0%
\raise.12em\copy0\hskip-.5\wd0\hbox to0pt{\hss$##1$\hss}\hskip.5\wd0}%
\def\lllink##1{\raise.40em\copy0\hskip-\wd0\raise.25em\copy0\hskip-\wd0%
\raise.10em\copy0\hskip-.5\wd0\hbox to0pt{\hss$##1$\hss}\hskip.5\wd0}%
\def\rootupright##1{\hbox to0pt{\raise.45em\copy1\hskip-.25em\raise1.3\ht1%
\hbox{\copy3\hskip.3em$##1$}\hss}%
\setbox6\hbox{\hskip.6em\copy1\copy1$##1##1$}%
\ifdim\MAXright<\wd6\MAXright\wd6\fi}%
\def\norootupright##1{\hbox to0pt{\raise.45em\copy1\hskip-.25em\raise1.3\ht1%
\hbox{\kern-.1em\copy4\hskip.3em$##1$}\hss}%
\setbox6\hbox{\hskip.6em\copy1\copy1$##1##1$}%
\ifdim\MAXright<\wd6\MAXright\wd6\fi}%
\def\rootdownright##1{\hbox to0pt{\raise-.5em\copy2\hskip-.25em\raise-1.35\ht1%
\hbox{\copy3\hskip.3em$##1$}\hss}\setbox6%
\hbox{\hskip.6em\copy2\copy2$##1##1$}%
\ifdim\MAXright<\wd6\MAXright\wd6\fi}%
\def\norootdownright##1{\hbox to0pt{\raise-.5em\copy2\hskip-.25em\raise-1.35\ht1%
\hbox{\kern-.1em\copy4\hskip.3em$##1$}\hss}\setbox6%
\hbox{\hskip.6em\copy2\copy2$##1##1$}%
\ifdim\MAXright<\wd6\MAXright\wd6\fi}%
\def\rootdown##1{\hbox to0pt{\hskip-.05em\vrule height.25em depth.65em%
\hskip-.25em\raise-.95em\hbox{\copy3\hskip.3em$##1$}\hss}%
\setbox6\hbox{$##1$}%
\ifdim\MAXright<\wd6\MAXright\wd6\fi}%
\def\norootdown##1{\hbox to0pt{\hskip-.05em\vrule height.25em depth.65em%
\hskip-.25em\raise-.95em\hbox{\copy4\hskip.3em$##1$}\hss}%
\setbox6\hbox{$##1$}%
\ifdim\MAXright<\wd6\MAXright\wd6\fi}%
\def\dots{\raise.25em\copy8\hskip.25em\raisebox{-0.025em}{$\cdots$}\hskip.10em%
\raise.25em\copy8}}%
\def\enddynkin{\ifdim\MAXright>1em\hskip.5\MAXright\else\hskip.5em\fi\egroup}%
\let\dyn\dynkin
\let\edyn\enddynkin
\theoremstyle{plain}
\newtheorem{thm}{Theorem}
\newtheorem{prop}{Proposition}[section]

\theoremstyle{definition}
\newtheorem{defn}[prop]{Definition}

\newtheorem{case}{Case}
\theoremstyle{remark}
\newtheorem{rem}[prop]{Remark}

\numberwithin{equation}{section}

\newcounter{numl}

\newcommand{\labelnuml}{\textup{(\roman{numl})}}
\newenvironment{numlist}{\begin{list}{\labelnuml}%
{\usecounter{numl}\setlength{\leftmargin}{0pt}%
\setlength{\itemindent}{2\parindent}%
\setlength{\itemsep}{\smallskipamount}
\def\makelabel ##1{\hss \llap {\upshape ##1}}}}{\end{list}}
\newcommand{\thismonth}{\ifcase\month\or
  January\or February\or March\or April\or May\or June\or
  July\or August\or September\or October\or November\or December\fi
  \space\number\year}
\DeclareSymbolFont{script}{U}{eus}{m}{n}
\DeclareSymbolFontAlphabet{\mathscr}{script}
\DeclareMathSymbol{\EuWedge}{0}{script}{"5E}
\DeclareMathAlphabet{\mathrmsl}{OT1}{cmr}{m}{sl}
\newcommand{\rssymb}[2]{\newcommand{#1}{{\mathrmsl{#2}}}}
\newcommand{\calsymb}[2]{\newcommand{#1}{{\mathcal{#2}}}}
\newcommand{\bbsymb}[2]{\newcommand{#1}{{\mathbb{#2}}}}
\newcommand{\liealg}[2]{\newcommand{#1}{{\mathfrak{#2}}}}
\newcommand{\liealr}[2]{\renewcommand{#1}{{\mathfrak{#2}}}}
\newcommand{\lieoper}[2]{\newcommand{#1}{\mathop
  {\mathfrak{#2}\null}\nolimits}}
\newcommand{\oper}[3][n]{\newcommand{#2}{\mathop
  {\mathrm{#3}\null}\ifx n#1\nolimits\else\limits\fi}}
\newcommand{\rsoper}[3][n]{\newcommand{#2}{\mathop
  {\mathrmsl{#3}\null}\ifx n#1\nolimits\else\limits\fi}}
\bbsymb\C{C} \bbsymb\F{F} \bbsymb\HQ{H}\bbsymb\I{I} \bbsymb\N{N} \bbsymb\OC{O}
\bbsymb\Q{Q} \bbsymb\R{R} \bbsymb\U{U} \bbsymb\V{V} \bbsymb\W{W} \bbsymb\Z{Z}
\calsymb\cA{A} \calsymb\cB{B} \calsymb\cC{C} \calsymb\cD{D} \calsymb\cE{E}
\calsymb\cF{F} \calsymb\cG{G} \calsymb\cH{H} \calsymb\cI{I} \calsymb\cJ{J}
\calsymb\cK{K} \calsymb\cL{L} \calsymb\cM{M} \calsymb\cN{N} \calsymb\cO{O}
\calsymb\cP{P} \calsymb\cQ{Q} \calsymb\cR{R} \calsymb\cS{S} \calsymb\cT{T}
\calsymb\cU{U} \calsymb\cV{V} \calsymb\cW{W} \calsymb\cX{X} \calsymb\cY{Y}
\calsymb\cZ{Z}

 \newcommand{\Gam}{{\mathrmsl\Gamma}}
\newcommand{\lam}{\lambda}
 \newcommand{\Sig}{{\mathrmsl\Sigma}}
\renewcommand{\geq}{\geqslant} \renewcommand{\leq}{\leqslant}
\rsoper\End{End} \rsoper\Hom{Hom}                
\rsoper\Sym{Sym} \rsoper\Skew{Skew}
\rsoper\gr{gr} \rsoper{\spn}{span}
\rsoper\Aut{Aut} \rsoper\Stab{Stab}              
\rsoper\GL{GL}\rsoper\SL{SL}\rsoper\PGL{PGL}\rsoper\PSL{PSL}\rsoper\Symp{Sp}
\rsoper\CO{CO}\rsoper\On{O} \rsoper\SO{SO}  \rsoper\Pin{Pin}\rsoper\Spin{Spin}
\rsoper\CU{CU}\rsoper\Un{U} \rsoper\SU{SU}
\rsoper\Diff{Diff} \rsoper\SDiff{SDiff}
\lieoper\der{der} \lieoper\stab{stab}            
\lieoper\gl{gl} \lieoper\sgl{sl}\lieoper\symp{sp}
\lieoper\co{co} \lieoper\so{so} \lieoper\spin{spin}
\lieoper\cu{cu} \lieoper\un{u}  \lieoper\su{su}
\rsoper\Vect{Vect} \rsoper\Ham{Ham}
\oper\real{Re}  
\oper\imag{Im}  
\newcommand{\ip}[1]{\langle#1\rangle}

\newcommand{\norm}[2][]{|\mkern-2mu|#2|\mkern-2mu|
  _{\lower1pt\hbox{${}_{#1}$}}}
\newcommand{\Norm}[2][]{\bigl|\mkern-3mu\bigr|#2\bigr|\mkern-3mu\bigr|
  _{\lower1pt\hbox{${}_{#1}$}}}

\newcommand{\restr}[1]{|_{#1}^{\vphantom{y}}}

\newcommand{\abrack}[1]{[\mkern-3mu[#1]\mkern-3mu]}

\newcommand{\transp}{^{\scriptscriptstyle\mathrm T\!}}
\newcommand{\act}{\mathinner{\raise2pt\hbox{$\centerdot$}}}
\newcommand{\dsum}{\oplus}                  
\newcommand{\Dsum}{\bigoplus}               
\newcommand{\tens}{\otimes}                 
\newcommand{\Wedge}{\EuWedge}               
\newcommand{\idealin}{\trianglelefteq}      
\newcommand{\into}{\hookrightarrow}         
\newcommand{\Proj}{\mathrmsl{P}}            
\newcommand{\RP}[1]{\R\Proj^{#1}}           
\newcommand{\CP}[1]{\C\Proj^{#1}}           
\newcommand{\HP}[1]{\HQ\Proj^{#1}}          
\newcommand{\st}{\mathrel{|}}               
\newcommand{\sub}{\subseteq}                
\rsoper\dimn{dim}                           
\rsoper\rank{rank}                          
\rsoper\degree{deg}                         
\rsoper\kernel{ker}\rsoper\image{im}        
\rsoper\alt{alt}   \rsoper\sym{sym}         
\rsoper\Ad{Ad}     \rsoper\ad{ad}           
\rsoper\CoAd{CoAd} \rsoper\coad{coad}       
\rsoper\trace{tr}  \rsoper\trfree{tf}       
\rsoper\detm{det}                           
\rsoper\Vol{Vol}                            
\rssymb\vol{vol}                            
\rssymb\iden{id}                            
\liealg{\f}{f} \liealg{\g}{g} \liealg{\h}{h} \liealg{\n}{n} \liealg{\m}{m}
\liealg{\p}{p} \liealg{\q}{q} \liealr{\t}{t} \liealg{\z}{z}
\newcommand{\ccon}{\theta}
\newcommand{\trl}{\zeta}
\newcommand{\symprod}{\mathbin{\raise1pt\hbox{$\scriptstyle\bigcirc$}}}
\newcommand{\CD}{\mathscr D}
\newcommand{\bx}{{\boldsymbol x}}
\newcommand{\rk}{\ell}
\newcommand{\tabv}[2]{$\vcenter{\hbox{\strut$#1$}\hbox{\strut{$#2$}}}$}
\newcommand{\tabvvv}[3]{$\vcenter{\hbox{\strut$#1$}\hbox{\strut{$#2$}}%
\hbox{\strut{$#3$}}}$}
\begin{document}
\title[Subriemannian metrics and the metrizability of parabolic
geometries]{Subriemannian metrics\\
and the metrizability of parabolic geometries}
\date{\today}
\author{David M.J. Calderbank}
\address{Mathematical Sciences\\ University of Bath\\
Bath BA2 7AY\\ UK.}
\email{D.M.J.Calderbank@bath.ac.uk}
\author{Jan Slov\smash{\'a}k}
\address{Department of Mathematics and Statistics\\
Masaryk University\\ Kotl\'a\v rsk\'a 2\\ 611 37 Brno\\ Czech Republic.}
\email{slovak@math.muni.cz}
\author{Vladim\smash{\'\i}r Sou\smash{\v c}ek}
\address{Mathematical Institute\\ Charles University\\ Sokolovsk\'a 83\\
186 75 Praha 8\\ Czech Republic.}
\email{soucek@karlin.mff.cuni.cz}
\begin{abstract}
We present the linearized metrizability problem in the context of parabolic
geometries and subriemannian geometry, generalizing the metrizability problem
in projective geometry studied by R. Liouville in 1889. We give a general
method for linearizability and a classification of all cases with irreducible
defining distribution where this method applies. These tools lead to natural
subriemannian metrics on generic distributions of interest in geometric control
theory.
\end{abstract}
\thanks{The authors thank the Czech Grant Agency, grant nr. P201/12/G028,
for financial support.}
\maketitle

\section{Introduction}

Many areas of geometric analysis and control theory deal with distributions on
smooth manifolds, i.e., smooth subbundles of the tangent bundle. Let $\cH\leq
TM$ be such a distribution of rank $n$ on a smooth $m$-dimensional manifold
$M$. A smooth curve $c\colon [a,b] \to M$ ($a\leq b\in\R$) is called
\emph{horizontal} if it is tangent to $\cH$ at every point, i.e., for every
$t\in [a,b]$, the tangent vector $\dot c(t)$ to $c$ at $c(t)\in M$ belongs to
$\cH$. It is well known that, at least locally, any two points $x,y\in
M$ can be connected by a horizontal curve $c$ if and only if $\cH$ is
\emph{bracket-generating} in the sense that any tangent vector can be obtained
from iterated Lie brackets of sections of $\cH$.

This paper is concerned with bracket-generating distributions arising in
\emph{parabolic geometries}~\cite{CS}, which are Cartan--Tanaka geometries
modelled on homogeneous spaces $G/P$ where $G$ is a semisimple Lie group and
$P\leq G$ a parabolic subgroup. On a manifold $M$ equipped with such a
parabolic geometry, each tangent space is modelled on the $P$-module $\g/\p$,
and the socle $\h$ of this $P$-module (the sum of its minimal nonzero
$P$-submodules) induces a bracket-generating distribution $\cH$ on $M$. Simple
and well-known examples include projective geometry and (Levi-nondegenerate)
hypersurface CR geometry: in the former case, $\g/\p$ is irreducible and so
$\cH=TM$, but in the latter case $\cH$ is the corank one contact distribution
of the hypersurface CR structure.

A more prototypical example for this paper is when $\cH\leq TM$ is generic of
rank $n$ and corank $\frac12 n(n-1)$, i.e., $m=\frac12 n(n+1) = n+\frac12
n(n-1)$, and $[\Gam(\cH),\Gam(\cH)]=\Gam(TM)$. In this case the Lie bracket on
sections of $\cH$ induces an isomorphism $\Wedge^2\cH\cong TM/\cH$ and the
distribution is said to be \emph{free}. Any such manifold is a parabolic
geometry where $G=\SO(V)$ with $\dim V=2n+1$ and $P$ is the stabilizer of a
maximal ($n$-dimensional) isotropic subspace $U$ of $V$~\cite{DS}. Then
$\g/\p$ has socle $\h\cong U^*\otimes (U^\perp/U)$ with quotient isomorphic to
$\Wedge^2 \h$, and $\h\leq \g/\p$ induces the distribution $\cH\leq TM$ on
$M$.

While parabolic geometry is the main tool for the present work, our motivation
is subriemannian geometry, which concerns the following notion~\cite{Mont}.

\begin{defn} Consider an $m$-dimensional manifold $M$ with a given smooth
distribution $\cH\leq TM$ of constant rank $n$. A
(pseudo-)Riemannian metric $g$ on $\cH$ is called a \emph{horizontal}
or \emph{subriemannian metric on $M$}.
\end{defn}
Horizontal metrics are important in both geometric analysis and control
theory. Among the horizontal curves joining two points, it may be important to
find those which are optimal in some sense, for example those of shortest
length with respect to a horizontal metric. Horizontal metrics also allow for
the definition of a hypo-elliptic sublaplacian~\cite{JL}, allowing methods of
harmonic analysis to be applied. However, this raises the question: what is a
good choice of horizontal metric?

For the distribution $\cH$ on a parabolic geometry, there is a natural
compatibility condition that can be imposed. Indeed, one of the key features
of such a geometry is that it admits a canonical class of connections,
called \emph{Weyl connections}, which form an affine space modelled on
the space of $1$-forms.

\begin{defn} A horizontal metric on the distribution $\cH\leq TM$ induced
from a parabolic geometry $M$ is \emph{compatible} if it is covariantly
constant in horizontal directions with respect to some Weyl connection on $M$.
We say $M$ is \emph{\textup(locally\textup) metrizable} if there exists
(locally) a compatible horizontal metric.
\end{defn}

The metrizability problem has been studied for several classes of parabolic
geometry with $\cH=TM$, in particular, the case of real projective. These
examples exhibit several interesting features, which we seek to generalize to
all parabolic geometries---in particular to those with $\cH\neq TM$.

First, whereas the metrizability condition appears to be highly nonlinear, it
linearizes when viewed as a condition on the inverse metric on $\cH^*$
multiplied by a suitable power of the horizontal volume form. Secondly, this
linear equation is highly overdetermined, with a finite dimensional solution
space. Hence parabolic geometries admitting such horizontal metrics are rather
special. This has been used to extract detailed information about the
structure of the geometry~\cite{BDE,CEMN,DM,EM,Frost,Liouville,Sinjukov}.

If $\h$ is the socle of $\g/\p$, it is not generally the case that $S^2\h$ is
irreducible---indeed $\h$ itself need not be irreducible. In order to
generalize the studied examples, we introduce a condition on $P$-submodules
$B\leq S^2 \h$ containing nondegenerate elements, which we call the
\emph{algebraic linearization condition} (ALC). Our first main result
(Theorem~\ref{alt}) justifies this terminology by showing that for parabolic
geometries and $P$-submodules $B\leq S^2\h$ satisfying the ALC, there is a
bijection between compatible horizontal metrics and nondegenerate solutions of
an overdetermined first order \emph{linear} differential equation. (In fact,
if $\h$ is not irreducible we need a technical extra condition, which we call the
\emph{strong} ALC.)

Our second main result (Theorem~\ref{main}) is a complete classification of
all parabolic geometries and all $P$-submodules $B\leq S^2\h$ such that $\h$
is irreducible and $B$ satisfies the ALC. The classification exhibits two
nicely counterbalancing features. On the one hand, among parabolic geometries
with irreducible socle, those admitting $P$-submodules $B\leq S^2\h$
satisfying the ALC are rare. On the other hand, the list of examples is quite
long: we state the classification using three tables containing 14 infinite
families and 6 exceptional cases. Many of these examples invite further study
(see e.g.~\cite{P}).

The structure of the paper is as follows. In section~\ref{s:bg} we briefly
outline the main notions and tools of parabolic geometry, referring
to~\cite{CS} for details, but concentrating on examples. We also establish the
local metrizability of the homogeneous model. In section~\ref{s:mlp}, we
describe the linearization principle and prove Theorem~\ref{alt}. We give
examples, and in particular show how explicit formulae can be obtained not
only for the homogeneous model, but also for so-called \emph{normal
  solutions}.  Section~\ref{s:class} is devoted to the main classification
result. We conclude by giving examples (Theorem~\ref{more}) where the socle is
not irreducible.

\section{Background and motivating examples}\label{s:bg}

We work throughout with real smooth manifolds $M$, real Lie groups $P$ and
real Lie algebras $\p$ (e.g., we view $\GL(n,\C)$ as a real Lie group and
$\gl(n,\C)$ as a real Lie algebra).

A (real or complex) \emph{$P$-module} $W$ is a finite dimensional (real or
complex) vector space carrying a representation $\rho_W\colon P\to \GL(W)$;
$W$ is then also a $\p$-module, where $\p$ is the Lie algebra of $P$, i.e., it
carries a representation $\tilde\rho_W \colon\p\to\gl(W)$. We write $\xi\act
w$ for $\tilde\rho_W(\xi)(w)$. The \emph{nilpotent radical} of $\p$ is the
intersection $\n$ of the kernels of all simple $\p$-modules. It is an ideal in
$\p$ and the quotient $\p_0:=\p/\n$ is reductive. We let $P_0:=P/\exp\n$ be
the corresponding quotient group with Lie algebra $\p_0$.  Any $P$-module $W$
has a filtration
\begin{equation}\label{eq:alg-filt}
0=W^{(0)}<W^{(1)}< \cdots < W^{(k)}=W\quad\text{with}\quad \n\act W^{(j)}\leq
W^{(j-1)} \quad \forall\,j\in \{1,\ldots k\},
\end{equation}
by $P$-submodules, where $\n\act W^{(j)}$ is the span of all $\xi\act w$ with
$\xi\in\n$ and $w\in W^{(j)}$. We let $\gr(W):=\Dsum_{1\leq j\leq k}
W^{(j)}/W^{(j-1)}$, which is a $P_0$-module.

\subsection{Parabolic geometries and Weyl structures}

Let $P\leq G$ be a closed Lie subgroup of a Lie group $G$, whose Lie algebra
$\p\leq \g$ has nilpotent radical $\n\idealin\p$.

\begin{defn} A \emph{Cartan geometry} of type $G/P$ on a smooth manifold $M$
is a principal $P$-bundle $\cG \to M$ equipped with a $P$-equivariant $1$-form
$\ccon\colon T\cG \to \g$ such that $\ccon_p\colon T_p\cG\to \g$ is an
isomorphism for all $p\in \cG$, and $\ccon(X_\xi)=\xi$ for all $\xi\in \p$,
where $\xi\mapsto X_\xi$ is the infinitesimal $\p$ action on $\cG$. The
\emph{homogeneous model} is the Cartan geometry $G\to G/P$ equipped with the
Maurer--Cartan form of $G$.
\end{defn}

Any $P$-module $W$ induces a bundle $\cW:=\cG\times_P W\to M$. A
filtration~\eqref{eq:alg-filt} of $W$ induces a bundle filtration
$0=\cW^{(0)}<\cW^{(1)}< \cdots < \cW^{(k)}=\cW$ with $\gr(\cW):=\Dsum_{k\in
  \N} \cW^{(k)}/\cW^{(k+1)}\cong \cG_0\times_{P_0} \gr(W)$ where
$\cG_0:=\cG/\exp\n$ is a principal $P_0$-bundle.

In particular, taking $W=\g/\p$, the projection of $\ccon$ onto $\g/\p$
induces a bundle isomorphism $TM\to \cG\times_P \g/\p$. This $P$-module has an
inductively defined filtration 
\[
0=\h^{(0)}< \h^{(1)}<\cdots< \h^{(k)}=\g/\p,\quad\text{where}\quad
\h^{(j)}:=\{x\in \g/\p\st \forall\,\xi\in \n,\;\;\xi\act x\in \h^{(j-1)}\}.
\]
In particular $\h:=\h^{(1)}$ induces a distribution $\cH\leq TM$ on $M$. We
return to this in \S\ref{ss:filt}.

We specialize to the case that $G$ is a semisimple Lie group and $P$ is a
parabolic subgroup of $G$, meaning that the nilpotent radical of $\p$ is its
Killing perp $\p^\perp$ in $\g$. Then Cartan geometries of type $G/P$ are
called \emph{parabolic geometries} and have several distinctive features which
we briefly explain and illustrate in the examples below (see~\cite{CS} for
further details).

First, the Killing form of $\g$ induces a duality between $\p^\perp$ and
$\g/\p$, and hence on any parabolic geometry of type $G/P$, we have a natural
isomorphism $\cG\times_P \p^\perp\cong T^*M$ dual to the isomorphism $TM\cong
\cG\times_P \g/\p$.

Secondly, the principal $P_0$-bundle $\cG_0$ has a distinguished family of
principal connections called \emph{Weyl connections}. To see this, it is
convenient to fix a parabolic subalgebra $\p^{\mathrm{op}}$ \emph{opposite} to
$\p$ in the sense that $\g=\p^\perp\oplus\p^{\mathrm{op}}$.  This identifies
$P_0$ with a subgroup of $P$, and induces a decomposition of $P_0$-modules
\begin{equation}\label{eq:g-decomp}
\g=\m\oplus \p_0\oplus \p^\perp,
\end{equation}
where $\m\cong \g/\p$ is the nilpotent radical of $\p^{\mathrm{op}}$. A
\emph{Weyl structure} is a $P_0$-equivariant splitting $\iota\colon\cG_0\into
\cG$ of the projection $\cG\to\cG_0$ (i.e., a reduction of structure group of
$\cG$ to $P_0$); the corresponding Weyl connection is the $\p_0$-component of
$\iota^*\ccon$. Weyl structures (or connections) form an affine space modelled
on the space of $1$-forms on $M$.

\subsection*{Summary} A manifold $M$ with a parabolic geometry of type
$G/P$ comes equipped with: a filtration of the tangent bundle $TM$, a $G_0$
structure on $\gr(TM)$, and a distinguished class of $G_0$-connections (the
Weyl connections).

\smallbreak There are general results~\cite{CSc,CS} stating that these data
are often sufficient to determine the parabolic geometry.  Rather than explore
this in generality, we turn to examples.

\subsection{Projective parabolic geometries}\label{ss:proj-pg}

We begin with some examples in which $\p^\perp$ is abelian, hence the
filtration of $\g/\p$ is trivial and~\eqref{eq:g-decomp} is a $\Z$-grading of
$\g$ as a Lie algebra, with $\p_0$ in degree $0$ and $\m,\p^\perp$ in degree
$\pm1$ (also called a $|1|$-grading). There is thus a $P_0$-structure on $TM$
and an algebraic bracket $\abrack{\cdot,\cdot}$ on $TM\dsum \p_0(M)\dsum T^*M$
where $\p_0(M)\leq \gl(TM)$ is the bundle induced by $\p_0$. In this case a
Weyl connection induces a $P_0$-connection $\nabla$ on $TM$ and any other Weyl
connection is given (on vector fields $Y,Z$) by
\begin{equation}\label{eq:weyl-affine}
\hat\nabla_ZY =\nabla_ZY + \abrack{\abrack{Z,\Upsilon}, Y} =\nabla_ZY +
\abrack{Z,\Upsilon}\act Y
\end{equation}
for some $1$-form $\Upsilon$, and we write $\hat\nabla=\nabla+\Upsilon$ for
short.

\subsubsection*{Projective geometry\nopunct} in dimension $m$ may be viewed
as a parabolic geometry of type $G/P$ where $G=\PGL(m+1,\R)$ and $P$ is the
parabolic subgroup of block lower triangular matrices with blocks of sizes $m$
and $1$. Here $\m=\R^m$, $\p_0 = \gl(m,\R)$, and $\p^\perp=\R^{m*}$, and the
homogeneous model $G/P$ is $m$-dimensional real projective space $\RP m$.

On a parabolic geometry of this type, the $G_0$-structure carries no
information as $G_0\cong\GL(m,\R)$, but two Weyl connections $\nabla$ and
$\hat\nabla=\nabla+\Upsilon$ are related (on vector fields $Y,Z$) by
\begin{equation}\label{projective-change}
\hat\nabla_ZY = \nabla_ZY+\abrack{\abrack{Z,\Upsilon}, Y}=
\nabla_ZY+\Upsilon(Z)Y+\Upsilon(Y) Z.
\end{equation}
Using abstract indices we may write this as
\[
\hat\nabla_aY^b = \nabla_aY^b +  \Upsilon_aY^b+\Upsilon_cY^c\delta_a^b.
\]
Thus the connections $\nabla$ and $\hat\nabla$ have the same torsion and the
same geodesics (as unparametrized curves). Setting the torsion to zero, we
have that the Weyl connections form a projective class $[\nabla]$.

\subsubsection*{\textup(Almost\textup) c-projective geometry\nopunct} is a
complex analogue of projective geometry~\cite{CEMN,Hrdina,HS,Yoshi78} with
$G=\PGL(m+1,\C)$ and $P\leq G$ block lower triangular as in the projective
case, so the homogeneous model $G/P$ is complex projective space $\CP m$
viewed as a real homogeneous space. A parabolic geometry of this type on a
$2m$-manifold $M$ is given by an almost complex structure $J\in\gl(TM)$ and a
class $[\nabla]$ of connections preserving $J$ which differ by
\begin{equation*}
\hat\nabla_aY^b = \nabla_aY^b +  \Upsilon_aY^b - \Upsilon_cJ^c_aJ^b_dY^d
+\Upsilon_cY^c\delta_a^b - \Upsilon_cJ^c_dY^dJ_a^b,
\end{equation*}
This can be obtained from the real projective formula by substituting
$(1,0)$-forms $\Upsilon-iJ\Upsilon$ and $(1,0)$ vectors
into~\eqref{projective-change}.

\subsubsection*{\textup(Almost\textup) grassmannian geometries\nopunct} are
generalizations of real projective geometry with $G=\PGL(m+k,\R)$ and $P$
block lower triangular with blocks of size $m$ and $k$.  The homogeneous model
$G/P$ is the grassmannian of $k$-planes in $\R^{m+k}$. On a parabolic geometry
of this type, the $G_0$-structure is given by an identification of the tangent
space with the tensor product of two auxiliary vector bundles $E^*$ and $F$ of
ranks $k$ and $m$ (with $\Wedge^kE^*\simeq \Wedge^mF$). In abstract index
notation, we write $e_{A'}$ for a section of $E$ and $f_A$ for a section of
$F$, hence $Y^{A'}_A$ for a vector field and $\eta_{A'}^B$ for a one-form.

The Weyl connections are tensor products of connections on $E^*$ and
$F$ with fixed torsion, and the freedom in their choice is
(cf.~\cite[p. 514]{CS})
\begin{equation}\label{grassmannian_change}
\hat\nabla^{A}_{A'}Y^{B'}_B=\nabla^{A}_{A'}Y^{B'}_B +
\delta^{B'}_{A'}\Upsilon^A_{C'}Y^{C'}_B +\delta_{B}^A\Upsilon_{A'}^CY_C^{B'}.
\end{equation}

When $m=2\ell$ and $k=2$ there is an interesting related geometry obtained by
replacing $\PGL(2\ell+2,\R)$ by another real form of $\PGL(2\ell+2,\C)$,
namely $\PGL(\ell+1,\HQ)$. The homogeneous model is then quaternionic
projective space $\HP \ell$, and a parabolic geometry of this type is
an (almost) quaternionic manifold~\cite{HS}.

\subsection{Parabolic geometries on filtered manifolds}\label{ss:filt}

We now turn to the examples of greater interest to us, in which $\cH$ is a
proper subbundle of $TM$. In fact, in these examples, the geometry is often
entirely determined by the distribution $\cH$, as we now discuss.

Given a smooth manifold $M$ of dimension $m$, equipped with a distribution
$\cH=\cH^{(1)}\leq TM$ of rank $n$, the Lie bracket of sections of $\cH$ (as
vector fields) defines a bundle map $\Wedge^2\cH\to TM/\cH$ called the
\emph{Levi form} of $\cH$. If we assume the image of the Levi form has
constant rank, it defines a subbundle $\cH^{(2)}\leq TM$ with $\cH^{(2)}/\cH$
equal to the image.  We thus inductively define $\cH^{(j)}\leq\cH^{(j+1)}\leq
TM$ such that $\cH^{(j+1)}/\cH^{(j)}$ is the image of the Lie bracket
$\cH\tens\cH^{(j)}\to TM/\cH^{(j)}$. If we further assume $\cH$ is
\emph{bracket-generating} i.e., $\cH^{(k)}=TM$ for some $k\in \N$, then we
obtain a filtration
\[
0=\cH^{(0)}<\cH^{(1)}<\cdots <\cH^{(k)}=TM
\]
such that the Lie bracket of sections of $\cH^{(i)}$ and $\cH^{(j)}$ is a
section of $\cH^{(i+j)}$. The associated graded vector bundle $\gr(TM)$ is, at
each $x\in M$, a graded Lie algebra $\g_x$ called the \emph{symbol algebra} of
$\cH$ at $x$. We finally assume that the Lie algebras $\g_x$ are all
isomorphic to the same nilpotent radical $\m$ of a fixed parabolic subalgebra
$\p^{\mathrm{op}}$ in a semisimple Lie algebra $\g$. In many cases
$\p_0=\p\cap \p^{\mathrm{op}}$, where $\p$ and $\p^{\mathrm{op}}$ are opposite
in $\g$, is the full algebra of automorphisms of $\m$ (as a graded Lie
algebra), and, as discussed in~\cite{CSc,CS}, this suffices to equip $M$ with a
parabolic geometry of type $G/P$.

The decomposition~\eqref{eq:g-decomp} of $\g$ is no longer $|1|$-graded and
this complicates the description of Weyl connections considerably. However, if
we work only with \emph{horizontal} (or \emph{partial}) \emph{connections},
i.e., restrict the Weyl connections to covariant derivatives in $\cH$
directions only, then the theory is as simple as in the $|1|$-graded case: the
Lie bracket between $\m$ and $\p^\perp$ in $\g$ induces a Lie bracket between
$\h\leq \m$ and $\p^\perp/[\p^\perp,\p^\perp]\cong \h^*$ with values in
$\p_0$, and hence an algebraic bracket
$\abrack{\cdot,\cdot}\colon\cH\tens\cH^*\to \p_0(M)$. Any two Weyl connections
$\nabla$ and $\hat\nabla$ are related by
\[
\hat\nabla_Z v=\nabla_Z v +\abrack{Z,\Upsilon}\act v
\]
where $\Upsilon$ is a section of $\cH^*$, $Z$ is a section of $\cH$, and $v$
is a section of $\cG_0\times_{P_0} V$ for any $G_0$-module $V$. We write
$\hat\nabla\restr{\cH}=\nabla\restr{\cH}+\Upsilon$ for short.

\subsubsection*{Free distributions\nopunct} are parabolic geometries
with $G=\SO(n+1,n)$ and $P$ block lower triangular with blocks of sizes $n$,
$1$, $n$, where the inner product is defined on the standard basis
$e_0,e_1\ldots e_{2n}$ by $\ip{e_i,e_{n+1+i}}=\ip{e_n,e_n}=\ip{e_{n+1+i},e_i}
=1$ for $0\leq i\leq n-1$ and all other inner products zero,
see~\cite{DS}. The homogeneous model $G/P$ is the grassmannian of maximal
isotropic subspaces of $\R^{2n+1}$. Elements of the Lie algebra
$\g=\so(n+1,1)$ have the form
\begin{equation*}
\begin{pmatrix}
-A\transp&-\xi\transp&B \\ -\gamma\transp&0&\xi\\ C&\gamma&A
\end{pmatrix}
\end{equation*}
where $B\transp=-B$ and $C\transp=-C$. Here $A\in\gl(n,\R)\cong\p_0$, $\xi\in
\R^n\cong\h$, $\gamma\in\R^{n*}\cong\h^*$, $B\in \Wedge^2\R^n \cong\Wedge^2\h$
and $C\in\Wedge^2\R^{n*}\cong\Wedge^2\h^*$.

A parabolic geometry of this type on a manifold $M$ of dimension $\frac12
n(n+1)$ may be determined by a distribution $\cH$ of rank $n$ whose Levi form
$\Wedge^2\cH\to TM/\cH$ is an isomorphism, hence the term ``free
distribution''. The $P_0$-structure is no additional data, and Weyl
connections may be determined as $P_0$-connections $\nabla$ such that for any
sections $Y,Z$ of $\cH$, the projection of $\nabla_Z Y-\nabla_Y Z$ onto
$TM/\cH\cong \Wedge^2\cH$ is $X\wedge Y$. If
$\hat\nabla\restr{\cH}=\nabla\restr{\cH}+\Upsilon$ we then compute that
\begin{equation}\label{free-dist-change}
\hat\nabla_Z Y =\nabla_ZY +\Upsilon(Y)Z.
\end{equation}

\subsubsection*{Free CR or quaternionic CR distributions\nopunct}
are obtained by replacing $\so(n+1,n)$ with $\g=\su(n+1,n)$ or $\symp(n+1,n)$,
again with (complex or quaternionic) blocks of sizes $n$, 1, $n$, and $\p$
being block lower triangular~\cite{SchmalzS}. Elements of $\g$ now have the
form
\begin{equation*}
\begin{pmatrix}
-A^\dagger&-\xi^\dagger&B \\ -\gamma^\dagger&\mu&\xi\\ C&\gamma&A
\end{pmatrix}
\end{equation*}
where ${}^\dagger$ denotes the (complex or quaternionic) hermitian conjugate,
$B^\dagger=-B$, $C^\dagger=-C$ and $\overline\mu=-\mu$. We may thus compute,
using matrix commutators
\begin{equation*}\label{general-k-n-k}
\bigl[[\xi-\xi^\dagger,\gamma-\gamma^\dagger],\eta-\eta^\dagger\bigr] =
\bigl(\xi\gamma\eta + \eta(\gamma\xi - \xi^\dagger\gamma^\dagger)\bigr)-
\bigl(\xi\gamma\eta + \eta(\gamma\xi - \xi^\dagger\gamma^\dagger)\bigr)^\dagger.
\end{equation*}
Note that the order here is important in the quaternionic case.

A parabolic geometry of this type has a complex or quaternionic rank $n$
distribution $\cH$ for which the Levi form is complex or quaternionic skew
hermitian, inducing an isomorphism of $TM/\cH$ with such forms on $\cH$.  If
$\hat\nabla\restr{\cH}=\nabla\restr{\cH}+\Upsilon$ we then have, on
sections $Y,Z$ of $\cH$,
\[
\hat\nabla_Z Y =\nabla_ZY + Z\,\Upsilon(Y)+Y\,(\Upsilon(Z)
-\overline{\Upsilon(Z)}).
\]

\subsection{First BGG operators, local metrizability of the homogeneous model,
  and normal solutions}\label{ss:1BGG}\label{ss:mhm}

Let $\cG \to M,\,\ccon$ be a Cartan geometry of type $G/P$. The extension of
$\cG$ by the left action of $P$ on $G$ is a principal $G$-bundle $\tilde\cG =
\cG\times_P G$ with $G$-connection $\tilde \ccon\colon \tilde \cG\to \g$, and
(by construction) a reduction $\cG\sub\tilde\cG$ of structure group to $P$,
and this provides an alternative description of the Cartan geometry. It
follows that for any $G$-module $V$, there is a canonical induced linear
connection on $\cV=\cG\times_P V\cong \tilde \cG\times_G V$. These bundles are
called \emph{tractor bundles} and their sections \emph{tractors}.

In the parabolic case, the \emph{BGG machinery} of~\cite{CSS, CD} provides a
sequence of invariant linear differential operators between bundles induced by
$P$-modules associated to $V$. The first such operator is defined on the
bundle $\cG\times_P V/(\p^\perp\act V)\cong \cV/(T^*M\act\cV)$ and is
overdetermined.

When $M=G/P$ is the homogeneous model, the kernel of this \emph{first BGG
  operator} is in bijection with the space of parallel sections of the tractor
bundle $\cV$, and the solutions have an explicit polynomial expression in
normal coordinates.  In more detail, fix an opposite parabolic subalgebra
$\p^{\mathrm{op}}$ to $\p\leq\g$, inducing a
decomposition~\eqref{eq:g-decomp}.  Then $\exp\m\leq G$ is a unipotent
subgroup of $G$ which determines a reduction $\cG_0\cong P_0\exp\m\leq G$ of
the homogeneous model $G\to G/P$ to the structure group $P_0$ over the image
$M$ of $\exp\m$ in $G/P$, hence a Weyl connection over $M$, the \emph{normal
  flat Weyl connection}.

Now if $\cV=\cG_0\times_{P_0} V$ is induced by a $P_0$-module $V$, the Weyl
covariant derivative of sections can be defined as the differentiation of
$P_0$-equivariant $V$-valued functions on $\cG_0$ in the direction of the
constant vector fields with respect to the Weyl connection, and the subgroup
$\exp\m$ is tangent to all such constant vector fields.  Thus any constant
coordinate function $f:\exp\m\to V$, with $f(x)=f_0$ for all $x\in\exp\m$,
defines a covariantly constant section with values in $\cV$. In particular,
choosing any nondegenerate symmetric 2-form $g$ in $S^2\m^*$, the metric
defined by the constant $g$ in the normal flat coordinates is covariantly
constant with respect to the normal flat Weyl connection. Thus the homogeneous
model $G/P$ is locally metrizable. By~\cite{CGH}, such explicit formulae also
apply on general curved geometries to the so called \emph{normal solutions},
which are those induced by parallel sections of the corresponding tractor
bundle. We discuss this further in \S\ref{ss:mtb}.

\section{Metrizability and the linearization principle}\label{s:mlp}

\subsection{First order operators}\label{ss:firstorder}

In~\cite{SS}, the second and third authors developed a theory of invariant
first order linear operators for parabolic geometries, generalizing work of
Fegan~\cite{Fegan} in the conformal case (cf.~\cite[Appendix B]{Gauduchon}).

We first fix some notation. The Killing form of $\g$ induces a nondegenerate
invariant scalar product on $\p_0=\p/\p^\perp$, such that the decomposition
into the semisimple part $\p_0^{ss}=[\p_0,\p_0]$ and the centre $\z(\p_0)$ is
orthogonal. Thus any Cartan subalgebra of $\p_0\tens\C$ has an orthogonal
decomposition $\t = \t'\dsum\t_0$, where $\t'$ is a Cartan subalgebra of
$\p_0^{ss}\tens\C$ and $\t_0=\z(\p_0)\tens\C$.  Further, $\t^* = \t'^*\dsum
\t_0^*$ is the dual decomposition, hence is orthogonal with respect to the
induced scalar product on $\t^*$.  We write the corresponding decomposition of
a weight $\lambda\in \t^*$ as $\lambda = \lambda'+\lambda^0$. Let $\Sig_0$ be
the set of simple roots $\alpha$ of $\g$ whose root space $\g_\alpha$ is in
$\h^*\otimes \C$. The remaining simple roots have root spaces in
$\p_0\otimes\C$, and hence belong to $\t'^*$ (i.e., they vanish on $\t_0$).
Hence $\alpha^0$, for $\alpha\in\Sig_0$, form a basis for $\t_0^*$ (dual to
the basis of $\t_0$ formed by the fundamental coweights which belong to
$\t_0$).

Let $V_\lambda$ be an irreducible complex $\p_0$-module with highest weight
$\lambda=\lambda'+\lambda^0\in\t^*$, let $\alpha = \alpha'+\alpha^0\in\Sig_0$,
and let $\mu = \mu'+\mu^0$ be the highest weight of a component $V_\mu$ in the
tensor product $V_{\lambda}\tens V_\alpha$. The key observation from
\cite[Theorem 4.4]{SS} is that there is a first order invariant operator
between sections of the bundles induced by $V_\lambda$ and $V_\mu$ if and
only if the scalar expression
\begin{equation*}
c_{\lambda,\mu,\alpha} =
\tfrac12\bigl((\mu-\lambda,\mu+\lambda+2\rho')-(\alpha,\alpha+2\rho')\bigr)
\end{equation*}
vanishes, where $\rho'\in \t'^*$ is half the sum of the positive roots of
$\p_0$. We split this expression into contributions from $\t'^*$ and $\t_0^*$
using the fact that $\mu^0=\lambda^0+\alpha^0$. Thus
\begin{equation}\label{eq:cas}\begin{split}
c_{\lambda, \mu, \alpha} &= c_{\lambda',\mu',\alpha'}+\tfrac12\bigl(
(\alpha_0, 2\lambda^0+\alpha^0)-(\alpha^0,\alpha^0)\bigr)\\
&= c_{\lambda',\mu',\alpha'} + (\lambda^0, \alpha^0).
\end{split}\end{equation}
If we fix $\lambda',\alpha$ and $\mu'$, this decomposition provides one (real)
linear equation on the central weight $\lambda^0$. This establishes the
existence of many first order operators~\cite{SS}. Here we
exploit~\eqref{eq:cas} in a more specific way.

\begin{prop}\label{p:fos} Let $\lambda'\in \t'^*$ be the highest weight of
a $\p_0^{ss}$-module, and for each $\alpha\in \Sig_0$, let $\mu'_\alpha
\in\t'^*$ be the highest weight of an irreducible component of $V_{\alpha'}
\tens V_{\lambda'}$.  Then there is a unique central weight $\lambda^0\in
\t_0^*$ such that for all $\alpha\in\Sig_0$, there is an invariant linear
first order operator between sections of the bundles induced by $V_\lambda$
and $V_{\mu_\alpha}$, where $\lambda=\lambda'+\lambda^0$ and $\mu_\alpha
=\mu_\alpha'+\lambda^0+\alpha^0$.
\end{prop}

A particular case of this result arises when $\mu_\alpha'=\lambda'+\alpha'$ so
that $\mu_\alpha=\lambda+\alpha$ and $V_{\mu_\alpha}$ is the Cartan product of
$V_\lambda$ and $V_{\alpha}$. In this case, the unique $\lambda^0$ is such
that $(\lambda,\alpha)=0$ for all $\alpha\in\Sig_0$, so that $\lambda$ is a
dominant weight for $\g$ and the first order system is the first BGG operator
on the bundle induced by $V_\lambda$.

\subsection{The algebraic linearization condition}

Let $(\cG\to M,\ccon)$ be a parabolic geometry of type $(G,P)$ and let $\h$
be the socle of the $\p$-module $\g/\p$, whose central weights form a basis
of $\z(\p_0)^*$.  As we have seen, $\cG\times_P\h\sub\cG\times_P \g/\p\cong
TM$ defines a (bracket generating) ``horizontal'' distribution $\cH\sub TM$.
Our aim is to construct compatible subriemannian (or pseudo-riemannian)
metrics, i.e., pseudo-riemannian metrics $g$ on $\cH$ for which there exists
a horizontal metric Weyl connection (a Weyl connection $\nabla$ with
$\nabla_Z g=0$ for all horizontal vector fields $Z$).

Let $c\colon \h^*\otimes S^2\h\to \h$ be the natural contraction. We then
posit the following.

\begin{defn} A nontrivial $\p_0$-submodule $B\leq S^2\h$ satisfies the
\emph{algebraic linearization condition} (ALC) if and only if $B$ has
nondegenerate elements, and there exist $\p_0$-submodules $\h_i \leq\h$ and
$B_i\leq S^2\h_i$ ($i\in\{1,\ldots r\}$) with $\h=\bigoplus_{i=1}^r \h_i$ and
$B=\bigoplus_{i=1}^r B_i$ such that for each $i\in\{1, \ldots r\}$, $B_i$ is
irreducible, and for any $\alpha\in\Sig_0$ and any irreducible component $W$
of $B_i\tens\C$, $(V_\alpha\tens W)\cap(\ker c\tens\C)$ is irreducible or zero.
\end{defn}
\begin{rem}\label{r:ALC} Note that $\eta\in B$ is nondegenerate if and only if
the same is true for each component $\eta_i\in B_i$. The restrictions
$b_i\colon \h^*\otimes B_i \to \h_i$ of $c$ are then surjective, and so we may
write $\h^*\otimes B_i=\ker b_i\oplus\trl_i(\h_i)$ where $\trl_i\colon \h_i\to
\h^*\otimes B_i$ is a $\p_0$-invariant map with $b_i\circ\trl_i=\iden_{\h_i}$.
Since $B_i$ is irreducible, it must lie in a single weight space of $\t_0$,
with weight $-\alpha^0-\beta^0$ where $\alpha,\beta\in \Sig_0$; hence it is in
the image of $\h_\alpha\tens\h_\beta\to S^2\h\otimes\C$ for the corresponding
weight spaces and so $\h_i\otimes\C$ has at most two irreducible components.
\end{rem}

\subsection{The linearization principle}\label{ss:lp}

Suppose first for simplicity that $B\leq S^2\h$ is absolutely irreducible and
satisfies the ALC (so $\h$ has at most two irreducible components) and let
$\pi=\iden_{\h^*\otimes B}-\trl\circ b$ be the projection onto $\ker
(b\colon\h^*\otimes B\to \h)$. The linearization method constructs a
(pseudo-riemannian) metric on $\cH$, i.e., a nondegenerate section $g$ of
$S^2\cH^*$ from a weighted inverse metric, i.e., a section $\eta$ of
$S^2\cH\otimes \cL$ for some line bundle $\cL$. For this we suppose $\eta$ is
a section of $\cB\otimes \cL$, where $\cB=\cG\times_P B$ and $\cL$ is a line
bundle induced by a weight of $\z(\p_0)$. We write $b,\trl,\pi$ also for
the induced bundle homomorphisms (tensored by the identity on $\cL$) and
choose $\cL$ so that there is an invariant first order linear operator $\CD$
from $\Gamma(\cB\otimes\cL)$ to $\Gamma(\ker b)$ with
$\CD=\pi\circ\nabla\restr\cH$ for any Weyl structure $\nabla$. If $\dim B=1$,
then $\ker b=0$, so $\CD$ is the zero operator, and we take $\cL$ to be
trivial.  Otherwise $\cL,\CD$ are determined by Proposition~\ref{p:fos}.  Due
to the ALC, $\ker b$ is then a sum of Cartan products of summands of $\h^*$
and $B$, and the operator $\CD$ is the first BGG operator.

Solutions $\eta$ of the linear differential equation $\CD\eta=0$ are
characterized by the fact that for some (hence any) Weyl structure $\nabla$,
there is a section $X^\nabla$ of $\cH\otimes\cL$ such that
\begin{equation*}
\nabla\restr{\cH}\eta=\trl(X^\nabla).
\end{equation*}
Now suppose $\hat\nabla\restr\cH=\nabla\restr\cH+\Upsilon$ with $\Upsilon$
in $\cH^*$. Then for any $Z\in \Gamma\cH$, $\hat\nabla_Z\eta=
\nabla_Z\eta+\abrack{Z,\Upsilon}\act\eta$, and
$\abrack{\cdot,\Upsilon}\act\eta$ is in the image of $\trl$ by the
invariance of $\CD$. Hence by Schur's lemma and \S\ref{ss:firstorder}
(i.e.,~\cite{SS}):
\begin{equation*}
\abrack{\cdot,\Upsilon}\act\eta=
(\trl\circ b)(\abrack{\cdot,\Upsilon}\act\eta)
=(\trl\circ b)\Bigl({\textstyle\sum_{\alpha\in\Sig_0} \ell_\alpha \Upsilon_\alpha}
\tens\eta\Bigr)
\end{equation*}
for nonzero scalars $\ell_\alpha$, where $\Upsilon=\sum_{\alpha\in\Sig_0}
\Upsilon_\alpha$ with $\Upsilon_\alpha\in V_\alpha\sub \h^*\tens\C$. If we
define $\sharp_\eta(\Upsilon)=\sum_{\alpha\in\Sig_0}\ell_\alpha
b(\Upsilon_\alpha\tens\eta)$, we deduce that
\begin{equation*}
\hat\nabla\restr{\cH}\eta=\nabla\restr{\cH}\eta
+\trl(\sharp_\eta(\Upsilon)).
\end{equation*}
Now if $\eta$ is a nondegenerate solution of $\CD\eta=0$, with
$\nabla\restr{\cH}\eta=\trl(X^\nabla)$ for some Weyl connection $\nabla$ and
$X^\nabla\in\Gamma(\cH\otimes\cL)$, we may take $\Upsilon =
-\sharp_\eta^{-1}(X^\nabla)$ to obtain
\begin{equation*}
\hat\nabla\restr{\cH}\eta
=\trl(X^\nabla)+\trl(\sharp_\eta(\Upsilon))=0.
\end{equation*}
Hence $\eta$ is (inverse to) a horizontal compatible metric, up to the shift
of the weight via the line bundle $\cL$. Finally, the nondegenerate weighted
metric $\eta$ allows us to build a nonvanishing section $\sigma$ of the line
bundle $\Wedge^m \cH\otimes\cL^{m/2}$, where $m=\dim \h$, with
$\hat\nabla\restr{\cH}\sigma=0$. This line bundle cannot be trivial because
the central weight of $\cB\otimes\cL$ is not zero. If $\h$ is absolutely
irreducible, then $\Wedge^m\cH\otimes\cL^{m/2}\cong \cL^k$ for some nonzero
$k$, and then $\psi=(\sigma^{-1/k}\eta)^{-1}$ is a section of $\cB^*$ with
$\hat\nabla\restr{\cH}\psi=0$. Otherwise, we need to assume the central
weights of $\Wedge^m\cH$ and $\cL$ are linearly dependent. The most natural
way to achieve this is to suppose that the simple roots $\alpha,\beta$ with
$\h\otimes\C=\h_\alpha\oplus\h_\beta$ are related by an automorphism of the
Dynkin diagram of $\g$.

\begin{defn}  A $\p_0$-submodule $B\leq S^2\h$ satisfies the
\emph{strong algebraic linearization condition} (strong ALC) if and only if
$B$ satisfies the ALC with respect to $\p_0$-submodules $\h_i\leq\h$ such that
whenever $\h_i\otimes\C=\h_\alpha\oplus\h_\beta$ for $\alpha,\beta\in\Sig_0$,
there is an automorphism of the Dynkin diagram of $\g$ preserving $\Sig_0$ and
interchanging $\alpha$ and $\beta$.
\end{defn}

With this definition, the linearization method yields the following result.

\begin{thm}\label{alt} Let $B\leq S^2\h$ satisfy the strong ALC with respect
to $B = \bigoplus_{i=1}^r B_i$ and $\h = \bigoplus_{i=1}^r \h_i$. Then for all
$i\in\{1,\ldots r\}$ there are induced line bundles $\cL_i$ and invariant
first order linear operators $\cD_i$ acting on sections of $\cB_i \otimes
\cL_i$ such that there is a bijection between nondegenerate solutions
$\eta_i:i\in\{1,\ldots r\}$ of the equations $\cD_i (\eta_i) = 0$, and
nondegenerate sections $\psi$ of $\cB^*$ with $\nabla\restr{\cH} \psi = 0$ for
some Weyl connection $\nabla$.
\end{thm}
\begin{proof} Define $b_i,\trl_i$ as in Remark~\ref{r:ALC} so
that $\h^*\tens B_i =\ker b_i\dsum\trl_i(\h_i)$, let $\pi_i
=\iden_{\h^*\otimes B_i}-\trl_i\circ b_i$ be the projection onto $\ker b_i$,
and let $\Sig_0^i=\{\alpha\in\Sig_0: V_\alpha\sub \h_i^*\tens\C\}$.  We apply
the same ideas as in the absolutely irreducible case to each irreducible
component $V_{\lam'}$ of $B_i\tens\C$.  If $\dim V_{\lam'}\geq 2$ then the ALC
implies that $(V_\alpha\tens V_{\lam'})\cap(\ker b_i\tens\C)$ is irreducible
for all $\alpha\in\Sig_0$, hence Proposition~\ref{p:fos} provides a unique
$\lam^0$ so that there is an invariant first order operator between sections
of the bundles induced by $V_\lam$ and $\ker b_i\tens\C$.  If instead,
$\dim V_{\lam'}=1$, then $V_\alpha\tens V_{\lam'}$ is irreducible, and is
contained in $\ker b_i\tens\C$ unless $\alpha\in\Sig_0^i$. We thus
supplement~\eqref{eq:cas} by the equations $(\lam^0,\alpha^0)=0$ when
$\alpha\in\Sig_0^i$.

Since $B_i$ is irreducible, $B_i\tens\C$ is either irreducible or has two
irreducible components with conjugate weights.  Now the system of
equations~\eqref{eq:cas} and $(\lam^0,\alpha^0)=0$ that we impose to find
$\lam^0$ are conjugation invariant.  Hence in either case, we obtain a line
bundle $\cL_i$ and an invariant first order linear operator
$\CD_i:=\pi_i\circ\nabla\restr{\cH}$ on $\cB_i\otimes\cL_i$, so that any
section $\eta_i$ satisfies $\CD_i(\eta_i)=0$ if and only if
\begin{equation*} 
\nabla\restr{\cH}\eta_i=\trl_i(X_i^\nabla)
\end{equation*}
for a suitable section $X_i^\nabla$ of $\cH_i\otimes\cL_i$. Given such
sections $\eta_i$, let $\eta=\sum_{i=1}^r \eta_i$. By construction, the
operator $b_i\circ\nabla\restr{\cH}$ is not invariant on
$\cB_i\otimes\cL_i$. Hence by Schur's Lemma, there are \emph{nonzero} scalars
$\ell_\alpha$ such that
\begin{equation*}
\abrack{\cdot,\Upsilon}\act\eta= (\trl\circ b)
(\abrack{\cdot,\Upsilon}\act\eta)= (\trl\circ b)
\Bigl({\textstyle\sum}_{i=1}^r{\textstyle\sum}_{\alpha\in\Sig_0^i} \ell_\alpha
\Upsilon_\alpha \tens\eta_i\Bigr)=
\Bigl({\textstyle\sum}_{\alpha\in\Sig_0} \ell_\alpha\Upsilon_\alpha \tens\eta\Bigr),
\end{equation*}
where $\Upsilon=\sum_{\alpha\in \Sig_0} \Upsilon_\alpha$ as before. As before,
we define $\sharp_\eta(\Upsilon)=\sum_{\alpha\in\Sig_0}\ell_\alpha
b(\Upsilon_\alpha\tens\eta)$, so that if
$\hat\nabla\restr{\cH}=\nabla\restr{\cH}+\Upsilon$ then
\begin{equation*}
\hat\nabla\restr{\cH}\eta=\nabla\restr{\cH}\eta
+\trl(\sharp_\eta(\Upsilon)).
\end{equation*}
If $\eta$ is a nondegenerate then $\sharp_{\eta}$ is invertible, and so if
$\CD(\eta)=0$, i.e., $\CD_i(\eta_i)=0$ for all $i$, then we may set
$\Upsilon:=-\sharp_{\eta}^{-1}(X^\nabla)$, where $X^\nabla=\sum_{i=1}^r
X_i^\nabla$ to obtain $\hat\nabla\restr{\cH}\eta=0$.

Finally, taking volume forms of $\eta_i$ on $\cH_i$ for each $i$, we obtain
nonvanishing sections $\sigma_i$ of $\Wedge^{m_i}\cH_i\otimes\cL_i^{m_i/2}$ with
$\hat\nabla\restr{\cH}\sigma_i=0$. The weights of the $\sigma_i$ are
linearly independent, and the strong ALC ensures that the central weights of
the $\cL_i$ are linear combinations of the central weights of
$\Wedge^{m_j}\cH_j$, so for every $i$, we can solve the linear system
$\eta_i\otimes\bigotimes_j \sigma_j^{a_{ij}}\in \cB_i$, and hence, inverting
each component, obtain the section $\psi$ of $\cB^*$ as required. Since the
system is invertible, $\eta$ can be obtained from $\psi$ and its volume forms 
on each $\cH_i$.
\end{proof}

If only the ALC is assumed, then the proof yields, in place of horizontally
parallel metrics on $\cH$, horizontally parallel conformal structures on each
$\cH_i$ and horizontally parallel sections of some line bundles.

\subsection{Example: projective geometry}

Let us illustrate the metrizability procedure by showing how the well-known
example of projective geometry~\cite{DM,EM,Liouville,Sinjukov} fits into the
general method.  Here $\g=\sgl(n+1,\R)= \h\oplus\gl(\h)\oplus\h^*$ and $S^2\h$
is irreducible. Since $\h^*\tens S^2\h \cong \h\oplus (\h^*\tens_0 S^2\h)$,
where the second summand is the trace-free part (the Cartan product),
$B=S^2\h$ satisfies the ALC.  The class of covariant derivatives defining the
projective structure depends on an arbitrary $1$-form $\Upsilon_a$ and two of
them are related by~\eqref{projective-change}. Hence on a section $\varphi$ of
$\cB=S^2TM$, we have
\[
\abrack{Z,\Upsilon}\act\varphi=
2\Upsilon(Z)\varphi+Z\tens\varphi(\Upsilon,\cdot)+\varphi(\Upsilon,\cdot)\tens
Z
\]
for any vector field $Z$ and $1$-form $\Upsilon$. If we twist by the line
bundle $\cL$ induced by the $P_0$-module $L$ with highest weight $-2\omega_1$,
then for $\eta \in \Gamma(\cB\otimes\cL)$ and $\hat\nabla=\nabla+\Upsilon$, we
have
\[
\hat\nabla_Z\eta=\nabla_Z\eta+b(\Upsilon\tens\eta)\odot Z
\]
where $X\odot Z=X\tens Z+Z\tens X$ and
$b(\Upsilon\tens\eta)=\eta(\Upsilon,\cdot)$ is the natural contraction.  In
abstract indices this contraction of $\Upsilon_c\eta^{ab}$ is
$\Upsilon_a\eta^{ab}$ and hence
\[
\hat\nabla_c\eta^{ab}=\nabla_c\eta^{ab}
+\delta_c^a\Upsilon_d\eta^{bd} +\delta_c^b\Upsilon_d\eta^{ad}.
\]
We thus have an invariant first order operator acting on $\eta$ (a first
BGG operator) whose solutions satisfy
\[
\nabla_Z\eta=\langle Z,\trl(X^\nabla)\rangle=\tfrac{1}{n+1} X^\nabla\odot Z.
\]
for some section $X^\nabla$ of $TM\tens\cL$. (Here
$\trl(X)=\frac1{n+1}X\odot\iden$, or in abstract indices,
$\trl(X^a)=\frac{1}{n+1}\,\bigl(X^a\delta^b_c+X^b\delta^a_c\bigr)$ so that
$b(\trl(X))=X$.) Evidently $\eta$ is parallel for $\hat\nabla$ provided
$b(\Upsilon\tens\eta)=-\frac1{n+1} X^\nabla$, which we can solve for
$\Upsilon$ if $\eta$ is nondegenerate. Direct computation shows that
$\det(\eta)$ is a section of $\cL^{-2}$. So $g^{ab}:=\det(\eta)\eta^{ab}$ is a
nondegenerate section of $S^2TM$ and its inverse is parallel with respect to
$\hat{\nabla}$. In terms of the general theory herein, if
\[
\abrack{\cdot,\Upsilon}\act \eta=
\ell\,(\trl\circ b)(\Upsilon\otimes\eta),
\]
then
\[
b(\Upsilon\tens\eta)\odot Z = \abrack{Z,\Upsilon}\act \eta
= \ell\,\langle Z,(\trl\circ b)(\Upsilon\otimes\eta)\rangle=
\frac{\ell}{n+1} b(\Upsilon\otimes\eta)\odot Z,
\]
and so $\ell=n+1$. Hence $\sharp_\eta(\Upsilon)=(n+1) b(\Upsilon\otimes\eta)$
and the solution is $\Upsilon=-\sharp_\eta^{-1}(X_\nabla)$.

\subsection{The metric tractor bundle}\label{ss:mtb} As we have seen
in \S\ref{ss:mhm}, the homogeneous model $G/P$ is always locally metrizable,
and solutions in the kernel of a given first BGG operator are induced by
parallel sections of a corresponding \emph{metric} tractor bundle. In general,
if $M$ has nontrivial curvature, not all solutions to a linearized
metrizability problem will correspond to such parallel sections: as discussed
in \S\ref{ss:1BGG}, those that do are called \emph{normal solutions} and
exhibit special features. In particular, as shown in~\cite{CGH}, they are
always of a simple polynomial forms in normal coordinates, exactly as in the
homogeneous model. Thus the explicit formulae from the homogeneous case form
an ansatz for solutions in general.

Let us discuss this in the case of free distributions from~\S\ref{ss:filt}.
Here $\g=\so(n+1,n)= \Wedge^2\h\oplus\h\oplus
\gl(\h)\oplus\h^*\oplus\Wedge^2\h^*$, $\cB=S^2\h$ is irreducible and satisfies
the ALC, just as in the case of projective geometry.  In this case, however,
there is no need to twist by a line bundle, since by~\eqref{free-dist-change},
we already have
\[
\hat\nabla_Z\eta=\nabla_Z\eta+b(\Upsilon\tens\eta)\odot Z
\]
for any sections $Z$ of $\cH$ and $\eta$ of $S^2\h$, where
$\hat\nabla\restr{\cH}=\nabla\restr{\cH}+\Upsilon$ and $b$ is the natural
contraction. The solution of the linearized metrizability problem then
proceeds exactly as in the projective case, so we now consider the form of the
normal solutions.

The \emph{standard tractor bundle} is the bundle associated to the defining
representation $V$ of $G=\SO(n+1,n)$.  Explicitly, using the matrix
description in~\S\ref{ss:filt}, we may write elements of $V$ as column vectors
\[
v=\begin{pmatrix} \lambda^a\\ \tau\\ \ell_a
\end{pmatrix}
\]
on which the action of the nilpotent radical $\m$ of $\p^{\mathrm{op}}$ is
given by
\[
\bx\act v=\begin{pmatrix}
0&x^a&y^{ab}\\0&0&-{x}^a\\
0&0&0
\end{pmatrix} 
\begin{pmatrix}
\lambda^b\\  \tau\\ \ell_b
\end{pmatrix}=
\begin{pmatrix}
x^a\tau+y^{ab}\ell_b\\ -{x}^b\ell_b\\ 0
\end{pmatrix}.
\]
The metric tractor bundle in this example is associated to the symmetric
tracefree square $S^2_0V$ of $V$. Elements of the symmetric square
$S^2V$ are given by
\[
\Phi=\begin{pmatrix}
\nu^{ab}\\
\sigma^b\\
\kappa\,|\,\psi^c_b\\
\xi_b\\
\tau_{bc}
\end{pmatrix},
\]
where $\nu^{ab}$ and $\tau_{bc}$ are symmetric, and such and element is in
$S^2_0V$ if $\kappa=-\psi^c_c$. Our convention is such that $\Phi=v\odot
\tilde v$ has components
\begin{gather*}
\nu^{ab}= \lambda^{a}\tilde\lambda^{b} + \tilde\lambda^{a}\lambda^{b};\quad
\sigma^b=\lambda^b\tilde\tau+\tau\tilde\lambda^b ;\quad
\kappa=\tau\tilde\tau;\quad
\psi_b^c=\ell_{b}\tilde\lambda^c+
\lambda^{c}\tilde\ell_{b} ;\\
\xi_b=\ell_{b}\tilde\tau+\tau\tilde\ell_{b} ;\quad
\tau_{bc}= \ell_{a}\tilde\ell_{b}+\tilde\ell_{a}\ell_{b}.
\end{gather*}
The action of the nilpotent radical on the symmetric square is given by
\[
\bx\act\Phi:=
\begin{pmatrix}
0&x^a&y^{ab}\\0&0&-{x}^a\\
0&0&0
\end{pmatrix}  
\begin{pmatrix}
\nu^{ab}\\
\sigma^b\\
\kappa\,|\,\psi^c_b\\
\xi_b\\
\tau_{bc}
\end{pmatrix}
=\begin{pmatrix}
x^{(a}\sigma^{b)}-y^{c(a}\psi^{b)}_c\\
{x}^c \psi_c^b+y^{bc}\xi_c-{x}^b\kappa\\
-{x}^b\xi_b\,|\,x^c\xi_b \\
-{x}^a\tau_{ab}\\
0
\end{pmatrix},
\]
where $x^{(a}\sigma^{b)}
=x^a\otimes\sigma^b+\sigma^b\otimes x^a$.
The iterated action is therefore given by
\begin{gather*}
\bx\act \bx\act\Phi=
\begin{pmatrix}
{x}^c x^{(a} \psi_c^{b)}+2x^{(a} y^{b)c}\xi_c
-x^{a}{x}^{b}\kappa \\
2{x}^c x^b\xi_c -y^{bc}{x}^a\tau_{ac}\\
{x}^b{x}^a\tau_{ab}\,|\,-x^c {x}^a\tau_{ab}\\
0\\
0
\end{pmatrix},\\
\bx\act \bx\act \bx\act\Phi=
\begin{pmatrix}
4x^{a} x^{b}x^c\xi_c
-2x^{(a} y^{b)c}{x}^d\tau_{dc} \\
-2x^b x^a x^c\tau_{ac}\\
0\,|\,0\\
0\\
0
\end{pmatrix},\qquad
\bx\act \bx\act \bx\act \bx\act\Phi=
\begin{pmatrix}
-4x^{a}{x}^{b}{x}^{c}{x}^{d}\tau_{cd}\\	
0\\
0\,|\,0\\
0\\
0
\end{pmatrix}
\end{gather*}
with all further iterates zero. The normal solution is the projection onto
$S^2\cH$ of $\exp(\bx)\cdot\Phi$, which is given by
\begin{multline*}
\eta^{ab}(x,y)=
\nu^{ab} +x^{(a}\sigma^{b)}-y^{c(a}\psi^{b)}_c
+\tfrac 12 {x}^c x^{(a} \psi_c^{b)}+x^{(a} y^{b)c}\xi_c
+\tfrac12 x^{a}{x}^{b}\psi^c_c\\
+\tfrac23 x^{a} x^{b}x^c\xi_c
-\tfrac13 x^{(a} y^{b)c}{x}^d\tau_{dc}
-\tfrac16 x^{a}{x}^{b}{x}^{c}{x}^{d}\tau_{cd}.
\end{multline*}

\section{Classification of metric parabolic geometries with irreducible $\h$}
\label{s:class}

We have seen that the linearizability problem of the existence of compatible
subriemannian metrics on parabolic geometries reduces to a purely algebraic
question related to the number of components in certain tensor products of the
$\p_0$-modules $\h$ and its dual $\h^*$. In fact, we are only interested in
the actions of the semisimple part of $\p_0=\p/\p^\perp$.
 
In this section, we classify all cases of the ALC where the defining
distribution of the parabolic geometry corresponding to $\h$ is
irreducible. This is the case with all $|1|$-graded geometries, but many
$|2|$-graded and some more general geometries are involved too. In order to
keep the story short, while still providing a complete and simple picture, we
use the schematic description of the chosen type of parabolic subalgebra $\p$
of $\g$ by crosses on the Dynkin diagram for $\g$ and we write weights of
$\p$-modules as linear combinations of the fundamental weights for $\g$,
depicted as the nonzero coefficients over the nodes of the diagrams, ignoring
those over the crossed nodes (see e.g. \cite[\S3.2]{CS} for these
conventions). This exactly provides the complete information on the
representation of the semisimple part of $\p_0$ in the case of complex
algebras and we always add further information on specific real forms of them.
Actually for practical reasons (and in accordance with common practice), we
rather write the weights of the dual $\p_0$-modules over the Dynkin diagram.
Moreover, the displayed diagrams and weights always correspond to the
complexified versions and thus we have to keep in mind their meaning for
particular real forms.

The classification is given in the following theorem. In the proof we also
describe the geometric properties of the metrics in any admissible component
$B$, mostly in terms of special structure related to the given parabolic
geometry. The classification in Table 1 was also obtained in \cite{P}.

\begin{thm}\label{main} Let $\p$ be a parabolic subalgebra in a real simple
Lie algebra $\g$ and let $B$ be a $\p$-submodule of $S^2\h$, with
$\h\cong(\p^\perp/[\p^\perp,\p^\perp])^*$ irreducible. Then $B$ satisfies the
ALC and admits nondegenerate elements if and only if one of the following
holds\textup:
\begin{itemize}
\item $\g$ is complex and the complexification of $(\p,B)$ appears in
  Table~\textup{\ref{t:hermitian};}
\item $(\g,\p,B)$ appears as a real
form in Table~\textup{\ref{t:absirred}} or~\textup{\ref{t:red};} 
\item $(\g,\p,B)$ is \textup(the underlying real Lie algebra of) the
complexification of a triple appearing in Table~\textup{\ref{t:absirred}}.
\end{itemize}
\begin{table}[!ht]
\begin{tabular}{|l|l|l|l|l|}
\hline
Case& Diagram $\Delta_\rk$ for $\p,B$ & Real simple $\g$ & Growth \\
\hline
$A_\rk^{h}$&$\dyn \root{1\strut}\link\root{}\dots\noroot{}\edyn\;\;
\dyn \noroot{}\link\root{}\dots\root{1}\edyn$&
$\sgl(\rk+1,\C)\;\; \rk\geq 2$ & $2\rk$\\
\hline
$B_\rk^{h}$&
$\dyn \root{1\strut}\link\root{}\dots\root{}\llink>\noroot{}\edyn\;\;
\dyn \noroot{}\llink<\root{}\link\root{}\dots\root{1}\edyn$&
$\so(2\rk+1,\C)\;\; \rk\geq 2$ &$2k, 2k+k(k-1)$\\
\hline
$G_2^{h}$&$\dyn \noroot{}\lllink<\root{1\strut}\edyn\;\;
\dyn \root{1}\lllink>\noroot{}\edyn$& $G_2^\C$&$4,6,10$\\
\hline
\end{tabular}
\smallbreak
\caption{Complex geometries with hermitian $B$}\label{t:hermitian}
\end{table}
\begin{table}[!ht]
\begin{tabular}{|l|l|l|l|l|}
\hline
Case& Diagram $\Delta_\rk$ for $\p,B$ & Real simple $\g$ & Growth \\
\hline
$A_\rk^{1,1}$&$\dyn \noroot{}\link\root{}\dots\root{2\strut}\edyn$&
$\sgl(\rk+1,\R)\;\; \rk\geq 2$ & $\rk$\\
\hline
$A_\rk^{1,2}$&$\dyn\root{}\link\noroot{}\link\root{}
\dots\root{}\link\root{1\strut}\link\root{}\edyn$& 
\tabv{\sgl(\rk+1,\R),\,\sgl(p+1,\HQ)}
{\rk=2p+1 ,\, p\geq2}&$4p$\\
\hline
$B_\rk^{1,k}$&$\dyn\root{2\strut}\link\root{}
\dots\root{}\link\nodroot{k\geq2}\link\root{}\dots\root{}\llink>\root{}\edyn$&
\tabv{\so(p,q),\;k\leq p\leq q}{p+q=2\rk+1}&\tabv{d=k(2\rk-2k+1),}
{n=d+\frac12 k(k-1)}\\
\hline
$B_{\rk}^{1,\rk}$&$\dyn \root{2\strut}\link\root{}\dots\root{}\llink>\noroot{}\edyn$&
$\so(\rk,\rk+1)\; \rk\geq 2$&$k,k+\frac12 k(k-1)$\\
\hline
$C_4^{1,2}$&$\dyn \root{}\link\noroot{}\link\root{}\llink<\root{1\strut}\edyn$&
\tabv{\symp(8,\R)}{\symp(2,2)\;\;\symp(1,3)}&$8,11$\\
\hline 
$C_\rk^{1,k}$&$\dyn\root{}\link\root{1\strut}\dots
\root{}\link\nodroot{k=2j\geq 4}\link\root{}\dots\root{}\llink<\root{}\edyn$&
\tabv{\symp(2\rk,\R)\;\;\;\symp(p,q)}{\quad\;\rk=p+q,\; k \leq p\leq q}
&\tabv{d=k(2\rk-2k),}{n=d+\frac12 k(k+1)}\\
\hline
$D_\rk^{1,k}$&$\dyn \root{2\strut}\link\root{}
\dots\root{}\link\nodroot{k\geq 2}\link\root{}
\dots\root{}\rootupright{}\rootdownright{}\edyn$&
\tabv{\so(p,q)\qquad\quad \so^*(2\rk)}
{\begin{matrix}2\rk=p+q\\
k\leq p\leq q\end{matrix}\quad \begin{matrix}k=2j\\ k\leq \rk-2\end{matrix}}
&\tabv{d=k(2\rk-2k),}{n=d+\frac12 k(k-1)}\\
\hline
$E_6^{1,1}$& $\dyn\noroot{}\link\root{}\link\root{}\rootdown{}
\link\root{}\link\root{1\strut}\edyn$& $E_{6(6)}$, $E_{6(-26)}$ &$16$\\
\hline
$G_2^{1,1}$&$\dyn \noroot{}\lllink<\root{2\strut}\edyn$& $G_{2(2)}$&$2,3,5$\\
\hline
\end{tabular}
\smallbreak
\caption{Real geometries with absolutely irreducible $\h$}
\label{t:absirred}
\end{table}

\begin{table}
\begin{tabular}{|l|l|l|l|l|}
\hline
Case& Diagram $\Delta_\rk$ for $\p,B$ & Real simple $\g$ & Growth \\
\hline
$A_3^{2,1}$&$\dyn \noroot{}\link\root{2\strut}\link\noroot{}\edyn$&
$\su(1,3),\;\su(2,2)$&$4,5$\\
\hline
$A_\rk^{2,k}$&$\dyn\root{1\strut}\dots\root{}\link\nodroot{k\geq 2}\link\root{}
\dots\root{}\link\nodroot{\rk-k}\link\root{}\dots\root{1\strut}\edyn$&
\tabv{\su(p,q),\;k\leq p \leq q}{\rk=p+q-1\geq 4}
&\tabv{d=2k(\rk-2k+1),}{n=d+k^2}\\
\hline
$A_\rk^{2,h}$&
$\begin{matrix}
\dyn\root{}\link\noroot{}\link\root{}\link\root{\strut1}
	\dots\root{}\link\noroot{}\link\root{}\edyn \\
	\oplus\\
	\dyn\root{}\link\noroot{}\link\root{}\dots\root{1}
	\link\root{}\link\noroot{}\link\root{}
	\edyn
\end{matrix}$&
\tabv{\su(p,q),\;2\leq p\leq q}{\rk=p+q-1\geq 6}&$4(\rk-3),4(\rk-2)$\\
\hline
$A_{2k+1}^{2,s}$&
$\begin{matrix}
\dyn\root{}\link\root{\strut1}\dots\root{}\link\noroot{}\link\root{}
\link\noroot{}\link\root{}\dots\root{}\link\root{}\edyn\\
 \oplus\\
\dyn\root{}\link\root{}\dots\root{}\link\noroot{}\link\root{}
\link\noroot{}\link\root{}\dots\root{1}\link\root{}\edyn 
\end{matrix}$&
\tabvvv{\su(k,k+2),}{\su(k+1,k+1)}{\rk=2k+1\geq 7}
 &$4k,4k+k^2$
\\
\hline
$A_{2k}^{2,s}$&$\begin{matrix}{\dyn\root{2\strut}\link\root{}\dots\root{}\link
\noroot{}\link\noroot{}\link\root{}\dots\root{}\link\root{}\edyn}\\
\oplus\\
{\dyn \root{}\link\root{}\dots\root{}\link\noroot{}\link\noroot{}
\link\root{}\dots\root{}\link\root{\smash{2}}\edyn}\end{matrix}$
& \tabv{\su(k,k+1)}{\rk=2k\geq 4}&$2k,2k+k^2$\\
\hline
$D_\rk^{2,s}$&$\dyn \root{2}\link\root{}\dots\root{}
\norootupright{}\norootdownright{}\edyn$&
\tabv{\so(\rk-1,\rk+1)}{\so^*(2\rk),\;\rk=2j+1}&\tabv{d=2(\rk-1),}
{d+\frac12(\rk-1)(\rk-2)}\\
\hline
$D_\rk^{2,h}$&$\dyn \root{}\link\root{1}\dots\root{}
\norootupright{}\norootdownright{}\edyn$&
\tabv{\so(\rk-1,\rk+1)}{\so^*(2\rk),\; \rk=2j+1}&\tabv{d=2(\rk-1),}
{d+\frac12(\rk-1)(\rk-2)}\\
\hline
$E_6^{2,h\vphantom{{}^2}}$&$\dyn\noroot{}\link\root{}\link\root{}\rootdown{1\strut}
\link\root{}\link\noroot{}\edyn$ & $E_{6(2)}$&$16,24$\\
\hline
\end{tabular}
\smallbreak
\caption{Real geometries with $\h$ not absolutely irreducible}
\label{t:red}
\end{table}
\end{thm}

\begin{proof}[Outline of Proof]
In the gradings of the complex algebras $\g$ corresponding to parabolic
geometries, the number of irreducible components of $\h^*$ is equal to the
number of crosses in the Dynkin diagram describing the chosen parabolic
subalgebra. However, in the real forms of $\g$, there might be complex or
quaternionic components giving rise to two components in the
complexification. These two complex components have to be either conjugate (in
the complex case) or isomorphic (in the quaternionic case).

The latter observation reduces our quest to diagrams with two crosses placed
in a symmetric way. Indeed, more than two crosses cannot result in one
component, while asymmetric positions of the crosses inevitably yield two
complex components which are neither conjugate nor isomorphic. Moreover,
having two components in the complexified $\h$, we may ignore the symmetric
products of the individual parts in $S^2\h$, because there cannot be any
nondegenerate metrics there.

We first dispense with the case that $\g$ is complex but $B$ is not, so that
$B\otimes\C$ is irreducible in $\g\otimes\C\cong\g\oplus\g$ and the diagram
for $(\p,B)$ is invariant under the automorphism exchanging the two components
of the Dynkin diagram. Thus $B\otimes\C=\h_\alpha\otimes\h_\beta$ where
$\h\otimes\C=\h_\alpha\oplus\h_\beta$. Now the ALC is satisfied provided
$\h_\alpha\otimes \h_\alpha^*$ (and hence also $\h_\beta\otimes\h_\beta^*$)
has precisely two irreducible components as a representation of a component of
$\p_0\otimes\C$. Only the (dual) defining representations in type A have this
property, and so $\g$ must have type $A,B$ or $G$, where the nodes crossed
in $\g\otimes\C$ are end nodes corresponding to short simple roots. The
possibilities are listed in Table~\ref{t:hermitian}, covering the following
three cases:

\begin{case}[$A_{\rk}^h$] The c-projective geometries may be equipped with
distinguished hermitian metrics.
\end{case}
\begin{case}[$B_{\rk}^h$] The almost complex version of a free distribution of
rank $k$, may be equipped with distinguished hermitian metrics.
\end{case}
\begin{case}[$G_2^h$] The almost complex version of the
$(2,3,5)$-distributions may be equipped with distinguished hermitian metrics.  
\end{case}

We analyse the remaining real cases with irreducible $\h$ by
the Dynkin type of $\g$ in the following sections.
\end{proof}

\subsection{Proof of Theorem~\ref{main} when $\g$ has type $A_\rk$}

The case $\rk=1$ is trivial, so we assume $\rk\geq 2$, and first consider the
case of a single crossed node. If the crossed node is one of the ends of the
Dynkin diagram, the only real $\g$ is the split form, $\h$ and $S^2\h$ are
irreducible, and $B=S^2\h$ satisfies the ALC: when $\rk=2$,
\[
B\simeq \dyn \noroot{}\link\root{2}\edyn \qquad \h^*\otimes B \simeq
\dyn\noroot{}\link\root{3}\edyn\oplus \dyn\noroot{}\link\root{1}\edyn
\]
and when $\rk\geq 3$,
\begin{equation*}
B\simeq \dyn \noroot{}\link\root{}\dots\root{2}\edyn\qquad
\h^*\otimes B \simeq \dyn\noroot{}\link\root{1}\dots\root{2}\edyn\oplus 
\dyn\noroot{}\link\root{}\dots\root{1}\edyn.
\end{equation*}
These examples can be summarized in the following statement.

\begin{case}[$A^{1,1}_\rk$] Here $\g=\sgl(\rk+1,\R)$, $\rk\geq2$, $\h\cong\R^\rk$
and $B=S^2\h$. This is the most classical case of projective structures on
$\rk$-dimensional manifolds $M$, and nondegenerate sections of $\cB$ are
inverse to arbitrary pseudo-Riemannian metrics on $M$.
\end{case}

Suppose next that the cross is adjacent to one end of the diagram, with
$\rk\geq 3$. We then have $S^2\h= B\oplus B'$, where
\begin{gather*}
\h\simeq\dyn\root{1}\link\noroot{}\link\root{}\dots\root{1}\edyn  \qquad
\h^*\simeq \dyn\root{1}\link\noroot{}\link\root{1}\dots\root{}\edyn\\
B\simeq\dyn\root{}\link\noroot{}\link\root{}\dots\root{1}\link\root{}\edyn
\;(\rk\geq 4)\qquad
B'\simeq\dyn\root{2}\link\noroot{}\link\root{}\dots\root{2}\edyn
\end{gather*}
and $B$ is trivial for $\rk=3$ (when $\h\cong \h^*$). The tensor product
$\h^*\otimes B'$ decomposes into four irreducible components, except for the
real form $\su(2,2)$ when $\rk=3$, in which case there are only three
components. In any case, $B'$ does not satisfy the ALC.

In order for $B$ to have nondegenerate elements, $\rk$ must be odd, and for
$\rk=2p+1\geq 5$, $\h^*\otimes B\simeq
\dyn\root{1}\link\noroot{}\link\root{1}\link\root{}
\dots\root{1}\link\root{}\edyn \oplus
\dyn\root{1}\link\noroot{}\link\root{}\dots\root{1}\edyn$; thus the ALC holds
for $B$.

\begin{case}[$A^{1,2}_\rk$] For each $\rk=2p+1\geq 5$, there are two real forms.
When $\g\simeq \sgl(2p+2,\R)$, the geometries are the almost grassmannian
structures on manifolds $M$ of dimension $4p$, modelled on the grassmannian of
$2$-planes in $\R^{2p}$. The tangent bundle $TM$ is identified with a tensor
product $E\otimes F$, where $\rank E=2$, $\rank F=2p$, and the nondegenerate
metrics in $\cB$ are tensor products of area forms on $E$ and symplectic forms
on $F$. When $\g\simeq \sgl(p,\HQ)$, the geometries are almost quaternionic
geometries, where $TM$ is a quaternionic vector bundle, and the nondegenerate
metrics in $\cB$ are the (real parts of) quaternionic hermitian forms.
\end{case}

When the cross is further from the ends of the diagram, we have $S^2\h=
B\oplus B'$ with
\begin{gather*}
B\simeq\dyn\root{}\link\root{1}\dots\root{}\link\noroot{}\link\root{}
\dots\root{1}\link\root{}\edyn\qquad
B'\simeq\dyn\root{2}\dots\root{}\link\noroot{}\link\root{}\dots\root{2}\edyn.
\end{gather*} 
and there are too many components in both $\h^*\otimes B$ and $\h^*\otimes
B'$ to satisfy the ALC.

We now turn to cases with two crossed nodes, related by the diagram
automorphism of $A_\rk$. First suppose the crossed nodes are the endpoints.
In order to have nontrivial $B$ we must have $\rk\geq3$, in which case
$S^2\h= B\oplus B'\oplus B''$ where
\begin{gather*}
\h\simeq \dyn\noroot{}\link\root{1}\dots\root{}\link\noroot{}\edyn\oplus
\dyn\noroot{}\link\root{}\dots\root{1}\link\noroot{}\edyn \simeq \h^*\\
B\simeq \dyn\noroot{}\link\root{2}\link\noroot{}\edyn \text{ or }
\dyn\noroot{}\link\root{1}\link\root{}
\dots\root{}\link\root{1}\link\noroot{}\edyn\qquad
B'\simeq \dyn\noroot{}\link\root{2}\dots\root{}\link\noroot{}\edyn
\oplus\dyn\noroot{}\link\root{}\dots\root{2}\link\noroot{}\edyn
\end{gather*}
and $B''$ is trivial. Clearly $\h^*\otimes B'$ has too many irreducible
components to satisfy the ALC, no matter which real form we consider.

It remains to consider $B$, first in the case $\rk=3$, where the possible
real forms (with $\h$ irreducible) are $\su(2,2)$ and $\su(1,3)$. Then
\[
\h^*\otimes B\simeq \bigl(\,\dyn\noroot{}\link\root{3}\link\noroot{}\edyn
\oplus \dyn\noroot{}\link\root{3}\link\noroot{}\edyn\,\bigr)
\oplus
\bigl(\,\dyn\noroot{}\link\root{1}\link\noroot{}\edyn
\oplus \dyn\noroot{}\link\root{1}\link\noroot{}\edyn\,\bigr)
\]
and the ALC is satisfied, since these are complexifications of two complex
components for the real form in question. However, for $\rk\geq 4$, we find
that the product $\h^*\otimes B$ leads to complexifications with three
complex components, so the ALC is not satisfied. 
 
\begin{case}[$A^{2,1}_3$] Here $\g$ is
	$\su(2,2)$ or $\su(1,3)$, and $M$ has a CR structure, i.e., a contact distribution
$\cH$ equipped with a complex structure. The Levi form induces the
class of trivial parallel hermitian metrics (the Weyl connections
corresponding to the contact forms leave parallel both the complex structure
and the symplectic form, thus also the associated metric, and the
metrizability problem is trivial as in the conformal case). However, we now
see that there may also be interesting compatible subriemannian metrics on
$\cH\leq TM$ which are hermitian and tracefree with respect to the
Levi form.
\end{case}

Now suppose the crosses are not placed at the ends, say the left one at the
$k$-th position, $2\leq k$. Thus we consider the real forms $\su(p,q)$ with
$k\leq p\leq q$. We have
\begin{gather*}
\h \simeq 
\dyn\root{1}\dots\root{}\link\noroot{}\link\root{}
\dots\root{1}\link\noroot{}\link\root{}\dots\root{}\edyn \oplus
\dyn\root{}\dots\root{}\link\noroot{}\link\root{1}
\dots\root{}\link\noroot{}\link\root{}\dots\root{1}\edyn 
\\
\h^* \simeq 
\dyn\root{}\dots\root{1}\link\noroot{}\link\root{1}
\dots\root{}\link\noroot{}\link\root{}\dots\root{}\edyn \oplus
\dyn\root{}\dots\root{}\link\noroot{}\link\root{}
\dots\root{1}\link\noroot{}\link\root{1}\dots\root{}\edyn 
\end{gather*}
for $\rk>2k$ and
\begin{gather*}
\h \simeq 
\dyn\root{1}\dots\root{}\link\noroot{}\link\noroot{}
\link\root{}\dots\root{}\edyn \oplus
\dyn\root{}\dots\root{}\link\noroot{}\link\noroot{}
\link\root{}\dots\root{1}\edyn 
\\
\h^* \simeq
\dyn\root{}\dots\root{1}\link\noroot{}\link\noroot{}
\link\root{}\dots\root{}\edyn \oplus
\dyn\root{}\dots\root{}\link\noroot{}\link\noroot{}
\link\root{1}\dots\root{}\edyn 
\end{gather*}
for $\rk=2k$. In particular, we have $S^2\h\supset B$ where
\begin{equation*}
B\simeq \dyn\root{1}\dots\root{}\link\noroot{}\link\root{}
\dots\root{}\link\noroot{}\link\root{}\dots\root{1}\edyn,
\end{equation*}
which admits nondegenerate metrics and satisfies the ALC, with
\begin{align*}
\h^*\otimes B&\simeq
\bigl(\,\dyn\root{1}\dots\root{1}\link\noroot{}\link\root{1}
\dots\root{}\link\noroot{}\link\root{}\dots\root{1}\edyn \oplus
\dyn\root{1}\dots\root{}\link\noroot{}\link\root{}
\dots\root{1}\link\noroot{}\link\root{1}\dots\root{1}\edyn\,\bigr)\\
&\;{}\oplus\bigl(\,\dyn\root{}\dots\root{}\link\noroot{}\link\root{1}
\dots\root{}\link\noroot{}\link\root{}\dots\root{1}\edyn\oplus
\dyn\root{1}\dots\root{}\link\noroot{}\link\root{}
\dots\root{1}\link\noroot{}\link\root{}\dots\root{}\edyn\, \bigr)\\
\text{or}\qquad
\h^*\otimes B &\simeq \bigl(\,\dyn\root{1}\dots\root{1}\link\noroot{}\link
\noroot{}\link\root{}\dots\root{1}\edyn \oplus
\dyn\root{1}\dots\root{}\link\noroot{}\link
\noroot{}\link\root{1}\dots\root{1}\edyn\,\bigr)
\\
&\;{}\oplus\bigl(\,\dyn\root{}\dots\root{}\link\noroot{}\link
\noroot{}\link\root{}\dots\root{1}\edyn\oplus
\dyn\root{1}\dots\root{}\link\noroot{}\link
\noroot{}\link\root{}\dots\root{}\edyn\, \bigr).
\end{align*}

\begin{case}[$A^{2,k}_\rk$] Here $\g\simeq \su(p,q)$ with nodes $k$
and $\rk+1-k$ crossed, where $2\leq k\leq p \leq q, p+q=\rk+1$.  In these
geometries, $\cH\cong E\otimes F$, where $E$ is a complex vector bundle of
rank $k$, and the rank $(\rk-2k+1)$ complex vector bundle $F$ comes with a
hermitian form of signature $(p-k,q-k)$. The corank of $\cH\leq TM$ is $k^2$,
and the metrics on $\cH$ are the products of hermitian metrics on $E$ with the
given ones on $F$. When $\rk=2k$ (i.e., $F$ has rank $1$), $\g=\su(k,k+1)$
with the nodes $k,k+1$ are crossed. These are the free CR geometries with
complex structure on $\cH$ studied in \cite{SchmalzS} (where it is also
explained how complex structure arises on $\cH$).
\end{case}

The remaining components of $S^2\h$ do not satisfy the ALC, except in special
cases $k=2$, $2k=\ell$ and $2k+1=\ell$. In particular, when $k=2$,
\begin{equation*}
B'\simeq \dyn\root{}\link\noroot{}\link\root{}\link\root{1}
\dots\root{}\link\noroot{}\link\root{}\edyn 
\oplus
\dyn\root{}\link\noroot{}\link\root{}\dots\root{1}
\link\root{}\link\noroot{}\link\root{}
\edyn
\end{equation*}
satisfies the ALC (and is nontrivial for $\rk\geq 6$).

\begin{case}[$A^{2,h}_\rk$] Here $\g\simeq \su(p,q)$ with nodes $2$
and $\rk-1$ crossed, where $2\leq p\leq q$ and $\rk=p+q-1\geq 6$.  In this
geometry, $\cH\cong E\otimes F$, where $E$ is a complex vector bundle of
rank $2$, and $F$ is a complex vector bundle of rank $\rk-3$. The
corank of $\cH\leq TM$ is $4$. The eligible metrics are the
complex symmetric bilinear forms of the form of tensor product of two exterior
forms.
\end{case}	

When $2k=\rk$, we obtain $S^2\h= B\oplus B'$ where
\begin{gather*}
B'= \begin{matrix}
\dyn\root{2}\dots\root{}\link\noroot{}\link\noroot{}\link\root{}
\dots\root{}\edyn\oplus{}\\
\dyn\root{}\dots\root{}\link\noroot{}\link\noroot{}\link\root{}
\dots\root{2}\edyn\quad\end{matrix}
\end{gather*}
which admits nondegenerate metrics, and satisfies the ALC, with
\begin{gather*}
\h^*\otimes B' \simeq \bigl(\,\dyn\root{2}\dots\root{1}\link\noroot{}\link
\noroot{}\link\root{}\dots\root{}\edyn \oplus
\dyn\root{}\dots\root{}\link\noroot{}\link
\noroot{}\link\root{1}\dots\root{2}\edyn\,\bigr)\oplus{}\\
\qquad\qquad\quad \bigl(\,\dyn\root{2}\dots\root{}\link\noroot{}\link
\noroot{}\link\root{1}\dots\root{}\edyn \oplus
\dyn\root{}\dots\root{1}\link\noroot{}\link
\noroot{}\link\root{}\dots\root{2}\edyn\,\bigr)\oplus{}\\
\qquad\qquad\;\,\bigl(\,\dyn\root{1}\dots\root{}\link\noroot{}\link
\noroot{}\link\root{}\dots\root{}\edyn\oplus
\dyn\root{}\dots\root{}\link\noroot{}\link
\noroot{}\link\root{}\dots\root{1}\edyn\, \bigr).
\end{gather*}
\begin{case}[$A^{2,s}_{2k}$] This case is again the free CR geometry,
with $\g=\su(k,k+1)$, but the eligible metrics are the complex bilinear
metrics on $\cH$.
\end{case}
Similarly, when $\rk=2k+1$ with the $k$-th and $(k+2)$-nd nodes crossed,
\begin{equation*}
B'\simeq \dyn\root{}\link\root{1}\dots\root{}\link\noroot{}\link\root{}
\link\noroot{}\link\root{}\dots\root{}\edyn \oplus
\dyn\root{}\dots\root{}\link\noroot{}\link\root{}
\link\noroot{}\link\root{}\dots\root{1}\link\root{}\edyn 
\end{equation*}
satisfies the ALC.	
\begin{case}[$A^{2,s}_{2k+1}$] Here $\rk=2k+1,$ $\g$ is $\su(k,k+2),$
or $\su(k+1,k+1),$ with nodes $k$ and $k+2$ crossed. In this geometry,
$\cH\cong E\otimes F$, where $E$ is a complex vector bundle of
rank $k$, and $F$ is a complex vector bundle of rank $2$. The
codimension of $\cH\leq TM$ is $k^2$. The eligible metrics are the
complex symmetric bilinear forms of the form of tensor product of two exterior
forms.
\end{case}
We have now exhausted all possibilities, completing the proof in type A.

\subsection{Proof of Theorem~\ref{main} when $\g$ has type $B_\rk$}

In the type $B$ case, there are no complex or quaternionic modules to
consider, so the irreducible cases have one cross only. The unique grading of
length one is odd dimensional conformal geometry. In dimension three we then
have
\[
\h^* \simeq \dyn\noroot{}\llink>\root{2}\edyn \simeq \h\qquad
S^2\h\simeq \dyn\noroot{}\llink>\root{4}\edyn\oplus 
\dyn\noroot{}\llink>\root{}\edyn.
\] 
The trivial representation in $S^2\h$ corresponds to the trivial case of
metrics in the conformal class, which are excluded from our classification,
and choosing $B$ to be the other component leads to three components in
$B\otimes \h^*$, so the ALC fails. Similarly, for conformal geometries of
dimensions $2\rk-1\geq 5$ we obtain
\[
\h^* \simeq \dyn\noroot{}\link\root{1}\dots\root{}\llink>\root{}\edyn 
\simeq \h\qquad
S^2\h\simeq \dyn\noroot{}\link\root{2}\dots\root{}\llink>\root{}\edyn
\oplus  \dyn\noroot{}\link\root{}\dots\root{}\llink>\root{}\edyn.
\] 
As before, the trivial summand is excluded, and the other component fails the
ALC.

We turn now to Lie contact geometries, with the second node crossed.  For
$B_3$,
\[
\h^* \simeq \dyn\root{1}\link\noroot{}\llink>\root{2}\edyn 
\simeq \h\qquad
S^2\h= B\oplus B'\oplus B''\simeq 
\dyn\root{2}\link\noroot{}\llink>\root{}\edyn
\oplus\dyn\root{}\link\noroot{}\llink>\root{2}\edyn
\oplus\dyn\root{2}\link\noroot{}\llink>\root{4}\edyn.
\] 
Here, $B\otimes \h^* = \dyn\root{3}\link\noroot{}\llink>\root{2}\edyn\oplus
\dyn\root{1}\link\noroot{}\llink>\root{2}\edyn$ and satisfies the ALC.  The
other choices lead to too many components. For $B_\rk$ with $\rk\geq 4$, we
have instead
\begin{gather*}
\h^* \simeq
\dyn\root{1}\link\noroot{}\link\root{1}\dots\root{}\llink>\root{}\edyn 
\simeq \h \qquad S^2\h= B\oplus B'\oplus B'' \\ 
B\simeq
\dyn\root{2}\link\noroot{}\link\root{}\dots\root{}\llink>\root{}\edyn\qquad
B'\simeq \dyn\root{2}\link\noroot{}\link\root{2}\dots\root{}\llink>
\root{}\edyn\qquad
B''\simeq  \dyn\root{}\link\noroot{}\link\root{}
\link\root{1}\dots\root{}\llink>\root{}\edyn,
\end{gather*}
except that when $\rk=4$, $B''=\dyn\root{}\link\noroot{}\link\root{}
\llink>\root{2}\edyn$. Now we check that $B'\otimes \h^*$ has six components,
$B''\otimes\h^*$ has three components, but the ALC is again satisfied by $B$.
Lie contact geometries exist for $\g=\so(p,q)$ with $2\leq p\leq q$; $\h$ is the
tensor product of defining representations $\R^2$ of $\sgl(2,\R)$ and
$\R^{p+q-4}$ of $\so(p-2,q-2)$, and $B$ is the tensor product of a symmetric
form on $\R^2$ and the defining inner product of signature $(p-2,q-2)$ on
$\R^{p+q-4}$.  See \cite[\S4.2.5]{CS} for more details on these
geometries.

Next we consider $B_\rk$ with the cross on $k$-th position, $3\leq k\leq
\rk-1$; the outcome is quite similar to the Lie contact case. For $k\neq
\rk-1$, $S^2\h= B \oplus B'\oplus B''$, where
\begin{gather*}
\h^*\simeq\dyn\root{}\dots\root{1}\link\noroot{}\link\root{1}\dots
\root{}\llink>\root{}\edyn \qquad
\h\simeq\dyn\root{1}\dots\root{}\link\noroot{}\link\root{1}\dots
\root{}\llink>\root{}\edyn\\
B\simeq\dyn\root{2}\dots\root{}\link\noroot{}\link\root{}\dots
\root{}\llink>\root{}\edyn\\
B'\simeq\dyn\root{2}\dots\root{}\link\noroot{}\link\root{2}\dots
\root{}\llink>\root{}\edyn\qquad
B''\simeq
\dyn\root{}\link\root{1}\dots\root{}\link\noroot{}\link\root{}\link\root{1}
\dots\root{}\llink>\root{}\edyn\\
\h^*\otimes B\simeq \dyn\root{2}\dots\root{1}\link\noroot{}\link\root{1}\dots
\root{}\llink>\root{}\edyn \oplus
\dyn\root{1}\dots\root{}\link\noroot{}\link\root{1}\dots
\root{}\llink>\root{}\edyn,
\end{gather*}
so $B$ satisfies the ALC, but $B'$ and $B''$ do not. If $k=\rk-1$,
$S^2\h=B\oplus B'\oplus B''$ with
\begin{gather*}
\h^* \simeq \dyn\root{}\dots\root{1}\link\noroot{}\llink>\root{2}\edyn
\qquad \h \simeq \dyn\root{1}\dots\root{}\link\noroot{}\llink>\root{2}\edyn\\
B\simeq \dyn\root{2}\dots\root{}\link\noroot{}\llink>\root{}\edyn\qquad
B'\simeq \dyn\root{2}\dots\root{}\link\noroot{}\llink>\root{4}\edyn\qquad
B''\simeq \dyn\root{}\link\root{1}
\dots\root{}\link\noroot{}\llink>\root{2}\edyn
\end{gather*}
and again, $B$ satisfies the ALC, but $B'$ and $B''$ do not. These
$|2|$-graded geometries are modelled on the flag variety of isotropic
$k$-planes and exist for the real forms $\so(p,q)$ with $k\leq p\leq q$. We have
$\h\cong \R^k\otimes \R^{p+q-k}$ and $B$ corresponds to the tensor product of
a symmetric form on $\R^k$ with the defining inner product on $\R^{p+q-k}$.
\begin{case}[$B^{1,k}_\rk$] Here $\g\simeq \so(p,q)$ with $k\leq p\leq q$
and $p+q=2\rk+1$, and the geometries come equipped with the identification of
the horizontal distribution $\cH\leq TM$ with the tensor product
$E\otimes F$, where $E$ has rank $k$ and $F$ carries a metric of signature
$(p-k,q-k)$. The corank of $\cH\leq TM$ is $\frac12 k(k-1)$.
The metrics in $B$ are the tensor products of symmetric nondegenerate forms on
$E$ and the given metric on $F$.
\end{case}

Finally, we arrive at the cross at the very end. For $B_\rk$ with $\rk\geq
2$, we have
\begin{gather*}
\h^*\simeq\dyn\root{}\dots\root{1}\llink>\noroot{}\edyn\qquad
\h\simeq\dyn\root{1}\dots\root{}\llink>\noroot{}\edyn\qquad
B= S^2\h \simeq \dyn\root{2}\dots\root{}\llink>\noroot{}\edyn\\
\h^*\otimes B \simeq \dyn\root{3}\llink>\noroot{}\edyn\oplus
\dyn\root{1}\llink>\noroot{}\edyn (\rk=2)\qquad
\h^*\otimes B \simeq
\dyn\root{2}\dots\root{1}\llink>\noroot{}\edyn
\oplus 
\dyn\root{1}\dots\root{}\llink>\noroot{}\edyn (\rk\geq 3),
\end{gather*}
and the ALC is satisfied.
\begin{case}[$B^{1,\rk}_\rk$]  Here $\g$ is the split form $\so(\rk,\rk+1)$.
The geometries are the well known free distributions, cf.~\cite{DS}, with rank
$\rk$ horizontal distribution $\cH\leq TM$ of corank
$\frac12\rk(\rk-1)$. The metrics in $B$ are all nondegenerate metrics on
$\cH$.
\end{case}

\subsection{Proof of Theorem~\ref{main} when $\g$ has type $C_\rk$}

As with type $B_\rk$, we only have to consider cases with a single crossed
node. We begin with the first node crossed, corresponding to the well known
contact projective structures, with
\[
\h^*\simeq\dyn \noroot{}\link\root{1}\dots\root{}\llink<\root{}\edyn\simeq\h\;;
\]
we have discussed the lowest dimension three already as the $B_2$ case, which
coincides with the free distribution of rank two. For $\rk\geq 3$, the picture
changes since
\begin{gather*}
S^2\h\simeq \dyn
\noroot{}\link\root{2}\dots\root{}\llink<\root{}\edyn\simeq B\\
B\otimes \h^* \simeq
\dyn\noroot{}\link\root{3}\dots\root{}\llink<\root{}\edyn
\oplus \dyn \noroot{}\link\root{2}\link\root{1}\dots\root{}\llink<\root{}\edyn
\oplus \dyn \noroot{}\link\root{1}\dots\root{}\llink<\root{}\edyn
\end{gather*}
and thus the ALC fails.

Moving on to the second node, we obtain another well known family of examples:
the quaternionic contact geometries (for $\g\cong\symp(p,\rk-p)$, $1\leq p\leq
\rk/2$) or their split analogues (for $\g\cong\symp(2\rk,\R)$)---see
\cite[\S4.3.3]{CS}. For $\rk=3$, we have
\begin{gather*}
\h^*\simeq\dyn \root{1}\link\noroot{}\llink<\root{1}\edyn\simeq\h\qquad 
S^2\h=B' \oplus B''\quad\text{with}\quad
B'\simeq \dyn \root{2}\link\noroot{}\llink<\root{2}\edyn
\end{gather*}
and $B''$ trivial, while for $\rk\geq 4$, we have
\begin{gather*}
\h^*\simeq\dyn \root{1}\link\noroot{}\link\root{1}\dots
\root{}\llink<\root{}\edyn\simeq \h \qquad  S^2\h=B \oplus B'\oplus B''\\
B\simeq \dyn \root{}\link\noroot{}\link\root{}\llink<\root{1}\edyn \quad
\text{or}\quad \dyn\root{0}\link\noroot{}\link\root{}\link\root{1}\dots
\root{}\llink<\root{}\edyn
\qquad B'\simeq\dyn \root{2}\link\noroot{}\link\root{2}
\dots\root{}\llink<\root{}\edyn
\end{gather*}
and $B''$ trivial. Since $\h^*\otimes B'$ decomposes into four components,
there are only nontrivial possibilities for $\rk\geq 4$. For $\rk=4$,
\begin{gather*}
\h^*\otimes B \simeq
\dyn \root{1}\link\noroot{}\link\root{1}\llink<\root{1}\edyn 
\oplus
\dyn \root{1}\link\noroot{}\link\root{1}\llink<\root{}\edyn
\end{gather*}
and so the ALC holds for $B$, but for $\rk\geq 5$, $ \h^*\otimes B$ has three
irreducible components, and the ALC is not satisfied.

\begin{case}[$C^{1,2}_4$] Here the possible real Lie algebras
are $\symp(8,\R)$, $ \symp(2,2)$, or $\symp(1,3)$, with the second node
crossed. In the first case, the geometries come equipped with the
identification of the horizontal distribution $\cH\leq TM$ with the
tensor product $E\otimes F$, where $E$ is rank $2$ and the rank $4$ vector
bundle $F$ comes with a symplectic form.  The eligible metrics in $B$ are the
tensor products of a area form on $E$ and the given symplectic form on $F$. In
the quaternionic cases, $\cH$ is quaternionic and the eligible metrics in $B$
are quaternionic hermitian forms.
\end{case}

Let us next suppose that the $k$-th node is crossed for $3\leq k\leq \rk-2$.
Then
\begin{gather*}
\h^*\simeq\dyn\root{}\dots\root{1}\link\noroot{}\link\root{1}\dots\root{}
\llink<\root{}\edyn
\qquad \h\simeq
\dyn\root{1}\dots\root{}\link\noroot{}\link\root{1}\dots\root{}
\llink<\root{}\edyn\\
S^2\h\simeq B \oplus B'\oplus B''\qquad
B\simeq \dyn\root{}\link\root{1}\dots\root{}\link\noroot{}
\link\root{}\dots\root{}\llink<\root{}\edyn\\
B'\simeq\dyn\root{2}\dots\root{}\link\noroot{}\link\root{2}\dots\root{}
\llink<\root{}\edyn\qquad
B''\simeq \dyn\root{}\link\root{1}\dots\root{}\link\noroot{}
\link\root{}\link\root{1}\dots\root{}\llink<\root{}\edyn\\
\h^*\otimes B \simeq
\dyn\root{}\link\root{1}\dots\root{1}\link\noroot{}\link\root{1}\dots
\root{}\llink<\root{}\edyn \oplus \dyn\root{1}\dots
\root{}\link\noroot{}\link\root{1}\dots\root{} \llink<\root{}\edyn
\end{gather*}
and so $B$ satisfies the ALC, but the other components do not. The relevant
metrics are again tensor products of an exterior form on the rank $k$
auxiliary bundle $E$ and the given symplectic form on $F$ (where the
horizontal distribution is identified with $E\otimes F$).  These geometries
are available for the split form $\symp(2\rk,\R)$ and, if $k$ is even then
also for the real forms $\symp(p,q)$, $k\leq p<q$.

The case with the cross at the last but one node is very similar. Here
\begin{gather*}
\h^*\simeq \dyn
\root{}\dots\root{}\link\root{1}\link\noroot{}\llink<\root{1}\edyn
\qquad\h\simeq \dyn
\root{1}\dots\root{}\link\root{}\link\noroot{}\llink<\root{1}\edyn\qquad
S^2\h =B\oplus B'\\
B\simeq \dyn\root{}\link\root{1}\dots\root{}\link\noroot{}\llink<\root{0}\edyn
\qquad B'\simeq
\dyn\root{2}\dots\root{}\link\root{}\link\noroot{}\llink<\root{2}\edyn\\
\h^*\otimes B\simeq \dyn
\root{}\link\root{2}\link\noroot{}\llink<\root{1}\edyn
\oplus \dyn
\root{1}\link\root{}\link\noroot{}\llink<\root{1}\edyn\quad (\rk=4)\\
\h^*\otimes B\simeq \dyn
\root{}\link\root{1}\dots\root{1}\link\noroot{}\llink<\root{1}\edyn
\oplus \dyn
\root{1}\link\root{}\link\noroot{}\llink<\root{1}\edyn\quad (\rk\geq 5)
\end{gather*}
and so $B$ satisfies the ALC (while $B$ does not).
\begin{case}[$C^{1,k}_\rk$] With the $k$-th node crossed
for $3\leq k\leq\rk-1$, the possible real Lie algebras are $\symp(2n,\R)$, and
if $k$ is even, then also $\symp(p,q)$, $k\leq p \leq q$. In the split case, the
horizontal distribution is a tensor product $\cH\simeq E\otimes F$ with
$E$ of rank $k$ and $F$ symplectic of rank $2\rk-2k$, the eligible metrics are
tensor products of antisymmetric forms on $E$ and the given symplectic form on
$F$.  In the quaternionic cases, $\cH$ comes with a quaternionic structure,
and the eligible metrics are quaternionic hermitian forms.
\end{case}

Finally, we consider the cross at the last node of $C_\rk$ with $\rk\ge 3$
($\rk=2$ is equivalent to the $B_2$ case with the first node crossed). In this
case
\begin{gather*}
\h^* \simeq \dyn \root{}\dots\root{2}\llink<\noroot{}\edyn
\qquad \h \simeq \dyn \root{2}\dots\root{}\llink<\noroot{}\edyn\\
S^2\h= B\oplus B' \qquad
B \simeq \dyn \root{}\link\root{2}\dots\root{}\llink<\noroot{}\edyn\qquad
B' \simeq \dyn \root{4}\dots\root{}\llink<\noroot{}\edyn
\end{gather*}
and both $B$ and $B'$ have too many components in their tensor products with
$\h^*$ to satisfy the ALC.

\subsection{Proof of Theorem~\ref{main} when $\g$ has type $D_\rk$}

We first consider the cases with one cross on $D_\rk$, $\rk\geq 4$, starting
with the the first node, i.e., the even dimensional conformal geometries,
where $S^2\h=B\oplus B'$ with
\[
\h^*\simeq \dyn
\noroot{}\link\root{1}\dots\root{}\rootupright{}\rootdownright{} \edyn \simeq
\h\qquad  B\simeq \dyn
\noroot{}\link\root{2}\dots\root{}\rootupright{}\rootdownright{} \edyn
\qquad B'\simeq \dyn
\noroot{}\link\root{}\dots\root{}\rootupright{}\rootdownright{} \edyn.
\]
As in the odd dimensional case (type $B_\rk$), $B$ does not satisfy the
ALC, and the trivial summand $B'$ yields metrics in the conformal class,
which we exclude.

We turn now to the Lie contact case, with the second node crossed. For
$\rk=4$,
\begin{gather*}
\h^*\simeq \dyn\root{1}\link\noroot{}\rootupright{1}\rootdownright{1}\edyn 
\simeq \h \qquad S^2\h= B\oplus B_1\oplus B_2 \oplus B_3 \\
B\simeq\dyn\root{2}\link\noroot{}\rootupright{2}\rootdownright{2}\edyn\qquad
B_1\simeq\dyn\root{2}\link\noroot{}\rootupright{}\rootdownright{}\edyn\qquad
B_2\simeq\dyn\root{}\link\noroot{}\rootupright{2}\rootdownright{}\edyn\qquad
B_3\simeq\dyn\root{}\link\noroot{}\rootupright{}\rootdownright{2}\edyn \\
B_1\otimes\h^*
\simeq \dyn\root{3}\link\noroot{}\rootupright{1}\rootdownright{1}\edyn 
\oplus \dyn\root{1}\link\noroot{}\rootupright{1}\rootdownright{1}\edyn 
\quad.
\end{gather*}
While $\h^*\otimes B$ has too many components, $B_1$ satisfies the ALC, as do
$B_2$ and $B_3$ by symmetry. The metrics are tensor products of two area forms
and a symmetric form on $\h\cong \R^2\otimes \R^2\otimes \R^2$. The geometries
exist for the real forms $\so(4,4)$, $\so(3,5)$ and the quaternionic
$\so^*(8)\simeq\so(2,6)$. Similarly, for $\rk\geq 5$, we have
\begin{gather*}
\h^* \simeq \dyn\root{1}\link\noroot{}\link\root{1}\dots
\root{}\rootupright{}\rootdownright{}\edyn \simeq \h\qquad
S^2\h = B\oplus B' \oplus B''\\
B\simeq\dyn\root{2}\link\noroot{}\link\root{}\dots
\root{}\rootupright{}\rootdownright{}\edyn \qquad
\h^*\otimes B \simeq \dyn\root{3}\link\noroot{}\link\root{1}\dots
\root{}\rootupright{}\rootdownright{}\edyn \oplus
\dyn\root{1}\link\noroot{}\link\root{1}\dots\root{}\rootupright{}\rootdownright{}\edyn \\
B'\simeq \dyn\root{2}\link\noroot{}\link\root{2}\dots
\root{}\rootupright{}\rootdownright{}\edyn \qquad
B''\simeq
\dyn\root{}\link\noroot{}\link\root{}\rootupright{1}\rootdownright{1}\edyn
\quad\text{or}\quad \dyn\root{}\link\noroot{}\link\root{}\link\root{1}\dots
\root{}\rootupright{}\rootdownright{}\edyn
\end{gather*}
where $B$ satisfies the ALC, but $B'$ and $B''$ do not. In addition to the
real forms $\so(p,q)$, $2\leq p\leq q$, $p+q=2\rk$, which are analogous to the
Lie contact geometries of type $B_\rk$, the real form $\so^*(2\rk)$ is also
possible.

As with type $B_\rk$, the cases where the $k$-th node is crossed, with $3\leq
k\leq\rk-2$ behave in a similar way. If $k\leq\rk-3$ then
\begin{gather*}
\h^* \simeq 
\dyn\root{}\dots\root{1}\link\noroot{}\link\root{1}
\dots\root{}\rootupright{}\rootdownright{}\edyn 
\qquad \h\simeq \dyn\root{1}\dots\root{}\link\noroot{}\link\root{1}
\dots\root{}\rootupright{}\rootdownright{}\edyn\qquad 
S^2\h=  B\oplus B' \oplus B''\\
B\simeq\dyn\root{2}\dots\root{}\link\noroot{}\link
\root{}\dots\root{}\rootupright{}\rootdownright{}\edyn\qquad
B'\simeq\dyn\root{2}\dots\root{}\link\noroot{}\link\root{2}
\dots\root{}\rootupright{}\rootdownright{}\edyn \\
B''\simeq\dyn\root{}\link\root{1}\dots\root{}\link\noroot{}
\link\root{}\rootupright{1}\rootdownright{1}\edyn
\quad\text{or}\quad
\dyn\root{}\link\root{1}\dots\root{}\link\noroot{}
\link\root{}\link\root{1}\dots\root{}\rootupright{}\rootdownright{}\edyn\\
\h^* \otimes B\simeq 
\dyn\root{2}\dots\root{1}\link\noroot{}\link
\root{1}\dots\root{}\rootupright{}\rootdownright{}\edyn 
\oplus \dyn\root{1}\dots\root{}\link\noroot{}
\link\root{1}\dots\root{}\rootupright{}\rootdownright{}\edyn 
\end{gather*}
so that $B$ satisfies the ALC, while $B'$ and $B''$ do not. The geometries
exist for real forms $\so(p,q)$, $k\leq p\leq q$, $p+q=2\rk$, and if $k$ is
even, then also for $\so^*(2\rk)$. If $k=\rk-2$, the computation differs
slightly, but the outcome is similar:
\begin{gather*}
\h^* \simeq \dyn\root{}\dots\root{1}\link\noroot{}
\rootupright{1}\rootdownright{1}\edyn \qquad\h\simeq 
\dyn\root{1}\dots\root{}\link\noroot{}
\rootupright{1}\rootdownright{1}\edyn \qquad
S^2\h =  B\oplus B_1 \oplus B_2\oplus B_3\\
B\simeq\dyn\root{2}\dots
\root{}\link\noroot{}\rootupright{}\rootdownright{}\edyn\qquad
\h^* \otimes B \simeq
\dyn\root{2}\dots\root{1}\link\noroot{}\rootupright{1}\rootdownright{1}\edyn 
\oplus
\dyn\root{1}\dots\root{}\link\noroot{}\rootupright{1}\rootdownright{1}\edyn\\
B_1\simeq\dyn\root{2}\dots
\root{}\link\noroot{}\rootupright{2}\rootdownright{2}\edyn\qquad
B_2\simeq\dyn\root{}\link\root{1}\dots
\root{}\link\noroot{}\rootupright{2}\rootdownright{0}\edyn \qquad
B_3\simeq\dyn\root{0}\link\root{1}\dots\root{}\link\noroot{}\rootupright{0}\rootdownright{2}\edyn
\end{gather*}
where $B$ satisfies the ALC, but the other cases do not. The geometries exist
for the real forms $\so(\rk,\rk)$, $\so(\rk-1,\rk+1)$, and if $\rk$ is even
then also $\so^*(2\rk)$.
\begin{case}[$D^{1,k}_\rk$] Here $\g\simeq \so(p,q)$ with $2\leq k\leq p\leq q$
and $p+q = 2n$, the geometries come equipped with the identification of the
horizontal distribution $\cH\leq TM$ with the tensor product
$E\otimes F$, where $E$ has rank $k$ and $F$ carries a metric of signature
$(p-k,q-k)$. The corank of $\cH\leq TM$ is $\frac12 k(k-1)$.  The
metrics in $B$ are the tensor products of symmetric nondegenerate forms on $E$
and the given metric on $F$. When $\g\simeq \so^*(2n)$ and $k$ is even, the
geometries come with the identification of the horizontal distribution
$\cH$ with the tensor product of a quaternionic rank $k$ bundle $E$ and
a quaternionic rank $n-2k$ bundle $F$ equipped with a quaternionic
skew-hermitian form. The metrics in $B$ are quaternionic hermitian forms.
\end{case}

The remaining case with one cross is the so called spinorial geometry with the
cross on one of the nodes in the fork. The case of $D_4$ coincides with the
$6$-dimensional conformal Riemannian geometry. For $\rk\geq 5$, we have
$S^2\h=B\oplus B'$ with
\begin{gather*}
\h^* \simeq \dyn \root{}\dots\root{1}\norootupright{}\rootdownright{}\edyn
\qquad \h\simeq
\dyn \root{}\link\root{1}\dots\root{}\norootupright{}\rootdownright{}\edyn
\qquad
B\simeq\dyn \root{}\link\root{2}\dots
\root{}\norootupright{}\rootdownright{}\edyn\\
B'\simeq
\dyn\root{}\link\root{}\link\root{}\norootupright{}\rootdownright{1}\edyn
\quad\text{or}\quad
\dyn \root{}\link\root{}\link\root{}\link\root{1}\dots
\root{}\norootupright{}\rootdownright{}\edyn.
\end{gather*}
Now $\h^*\otimes B$ has three summands, as does $\h^*\otimes B'$, except for
$\rk=5$, when
\begin{gather*}
\h^*\otimes B'\simeq
\dyn\root{}\link\root{}\link\root{1}\norootupright{}\rootdownright{1}\edyn
\oplus
\dyn\root{}\link\root{1}\link\root{}\norootupright{}\rootdownright{1}\edyn\;.
\end{gather*}
Here, in the complex setting, $\h\cong \Wedge^2\C^5$, and $B\cong\C^{5*}\cong
\Wedge^4\C^5\leq S^2\Wedge^2\C^5$, where $\alpha\in B\cong\C^{5*}$
determines a metric $g_\alpha$ on $\h^*\cong \Wedge^2\C^{5*}$ by
$g_\alpha(\xi,\eta)=\alpha\wedge\xi\wedge\eta\in \Wedge^5\C^{5*}\cong \C$.
Such a metric is never nondegenerate, so this case is excluded.

We next consider $D_\rk$ cases with two crossed nodes. For $\h$ to be
irreducible, the semisimple part of the Levi factor $\p/\p^\perp$ must be
simple.  Indeed, working in the complex setting, a direct check reveals that
breaking the Dynkin diagram by two crosses into more than one part always
leads to non-isomorphic representations for the two components of $\h$.
Furthermore, the only way to obtain isomorphic components is to take the two
spinorial nodes of the $D_\rk$ diagram. The only real forms compatible with
this geometry are the split form $\so(\rk,\rk)$, the quasi-split form
$\so(\rk+1,\rk-1)$ and the quaternionic form $\so^*(2\rk)$ with $\rk=2p+1$
odd.  In the split case, $\h$ is not irreducible, so this does not fit into
our classification. In the quasi-split case, $\h\cong \R^{\rk-1}\otimes_\R\C$
is complex, while in the quaternionic case, $\h$ is quaternionic. In $S^2\h =
B\oplus B'\oplus B''$, with
\begin{gather*}
\h^* = \dyn \root{}\dots\root{1}\norootupright{-2}\norootdownright{}\edyn
\oplus \dyn \root{}\dots\root{1}\norootupright{}\norootdownright{-2}\edyn
\\
\h= \dyn \root{1}\dots\root{}\norootupright{1}\norootdownright{-1}\edyn
\oplus \dyn \root{1}\dots\root{}\norootupright{-1}\norootdownright{1}\edyn
\\
B\cong\dyn \root{2}\dots\root{}\norootupright{}\norootdownright{}\edyn\qquad
B'\cong\dyn \root{}\link\root{1}\dots
\root{}\norootupright{}\norootdownright{}\edyn\qquad
B''\cong\dyn\root{2}\dots\root{}\norootupright{2}\norootdownright{-2}\edyn
\oplus\dyn\root{2}\dots\root{}\norootupright{-2}\norootdownright{2}\edyn
\end{gather*}
where we denote the nonzero weights over the crossed nodes for clarity.
Observe that
\begin{gather*}
\h^*\otimes B\simeq\Bigl(\,
\dyn \root{2}\dots\root{1}\norootupright{}\norootdownright{}\edyn\oplus
\dyn\root{2}\dots\root{1}\norootupright{}\norootdownright{}\edyn\;\Bigr)
\oplus \Bigl(\,
\dyn \root{1}\dots\root{}\norootupright{}\norootdownright{}\edyn \oplus
\dyn\root{1}\dots\root{}\norootupright{}\norootdownright{}\edyn\;\Bigr)\\
\h^*\otimes B'\simeq\Bigl(\,
\dyn \root{}\link\root{1}\dots\root{1}\norootupright{}\norootdownright{}
\edyn \oplus \dyn \root{}\link\root{1}\dots
\root{1}\norootupright{}\norootdownright{}\edyn\;\Bigr) \oplus \Bigl(\,
\dyn \root{1}\dots\root{}\norootupright{}\norootdownright{}\edyn
\oplus \dyn
\root{1}\dots\root{}\norootupright{}\norootdownright{}\edyn\; \Bigr)
\end{gather*}
so that both $B$ and $B'$ satisfy the ALC, but $B''$ does not ($\h^*\otimes
B''$ has eight components).
\begin{case}[$D^{2,s}_\rk$] When $\g=\so(\rk-1,\rk+1)$, the horizontal
distribution $\cH \leq TM$ is a complex vector bundle of complex rank
$\rk-1$, and the metrics in $B$ are complex bilinear. When $\g=\so^*(2\rk)$,
with $\rk=2p+1$ odd, the horizontal distribution $\cH\leq TM$ has a
quaternionic structure of quaternionic rank $p$ and the metrics in $B$
are quaternionic skew-hermitian.
\end{case}
\begin{case}[$D^{2,h}_\rk$] This case involves the same geometries as in the
previous case, with $\rk=2p+1$ odd, except that when $\g=\so(\rk-1,\rk+1)$,
the metrics in $B'$ are hermitian, while for $\g=\so^*(2\rk)$, the metrics in
$B$ are quaternionic hermitian.
\end{case}

\subsection{Proof of Theorem~\ref{main} when $\g$ has exceptional type}

The first case we consider is the Lie algebra $E_6$. Let us consider
possibilities for parabolic subalgebras with one crossed node.  The first
possibility is
\begin{gather*}
\h^* \simeq \dyn \noroot{}\link\root{1}\link\root{}\rootdown{}
\link\root{}\link\root{}\edyn
\qquad \h \simeq \dyn \noroot{}\link\root{}\link\root{}\rootdown{1}
\link\root{}\link\root{}\edyn\\
S^2\h =B\oplus B' \qquad
B\simeq \dyn \noroot{}\link\root{}\link\root{}\rootdown{}
\link\root{}\link\root{1}\edyn\qquad
B'\simeq \dyn \noroot{}\link\root{}\link\root{}\rootdown{2}
\link\root{}\link\root{}\edyn
\end{gather*}
The product $\h^*\otimes B$ decomposes (as the product of a spinor and
defining representation of SO$(10)$) into the sum of two $P$-modules, hence
the ALC is satisfied.
\begin{case}[$E^{1,1}_6$] This is the $|1|$-graded geometry for $E_6$ for which
the allowed real forms are the split form $E_{6(6)}$, or $E_{6(-26)}$, and
$\cH= TM$ carries the structure of basic spinor representation $S^+$ of
$\so(5,5)$, or $\so(1,9)$ respectively. The $P$-module $B$ corresponding
to the eligible metrics is the defining representation of $\so(5,5)$ or
$\so(1,9)$.
\end{case}
Consider next the adjoint variety, with the node on the short leg crossed.  We
have
\begin{gather*}
\h^* \simeq  \h \simeq \dyn \root{}\link\root{}\link\root{1}\norootdown{}
\link\root{}\link\root{}\edyn\\
S^2\h=B\oplus B'\qquad
B \simeq \dyn \root{}\link\root{}\link\root{2}\norootdown{}
\link\root{}\link\root{}\edyn\qquad
B' \simeq \dyn \root{1}\link\root{}\link\root{}\norootdown{}
\link\root{}\link\root{1}\edyn
\end{gather*}
and find that both $\h^*\otimes B$ and $h^*\otimes B'$ have four components.

In the remaining two cases with one crossed node, the semisimple part of
$\p/\p^\perp$ is not simple, and it is easy to see that the ALC cannot be
satisfied:
\begin{gather*}
\h^* \simeq \dyn \root{1}\link\noroot{}\link\root{1}\rootdown{}
\link\root{}\link\root{}\edyn
\qquad \h \simeq \dyn \root{1}\link\noroot{}\link\root{}\rootdown{}
\link\root{1}\link\root{}\edyn\\
S^2\h \simeq \dyn \root{2}\link\noroot{}\link\root{}\rootdown{}
\link\root{2}\link\root{}\edyn
\oplus \dyn \root{2}\link\noroot{}\link\root{}\rootdown{1}
\link\root{}\link\root{}\edyn 
\oplus \dyn \root{}\link\noroot{}\link\root{1}\rootdown{}
\link\root{}\link\root{1}\edyn\\
\h^* \simeq \dyn \root{}\link\root{1}\link\noroot{}\rootdown{1}
\link\root{1}\link\root{}\edyn
\qquad \h \simeq \dyn \root{1}\link\root{}\link\noroot{}\rootdown{1}
\link\root{}\link\root{1}\edyn\\
S^2\h\simeq \dyn \root{2}\link\root{}\link\noroot{}\rootdown{2}
\link\root{}\link\root{2}\edyn
\oplus \dyn \root{2}\link\root{}\link\noroot{}\rootdown{}
\link\root{1}\link\root{}\edyn\oplus \dyn \root{}\link\root{1}\link\noroot{}\rootdown{}
\link\root{}\link\root{2}\edyn
\oplus \dyn \root{}\link\root{1}\link\noroot{}\rootdown{2}
\link\root{1}\link\root{}\edyn.
\end{gather*}
The only case with two crosses for which $\h$ could be irreducible is
\begin{gather*}
\h^* \simeq  \h \simeq \dyn \noroot{}\link\root{1}\link\root{}\rootdown{}
\link\root{}\link\noroot{}\edyn
\oplus
\dyn \noroot{}\link\root{}\link\root{}\rootdown{}
\link\root{1}\link\noroot{}\edyn,
\end{gather*}
and indeed, $\h$ is irreducible for the quasi-split real form $E_{6(2)}$.  For
this real form, the nontrivial irreducible summands in $S^2\h$ are
\begin{gather*}
B\simeq  \dyn \noroot{}\link\root{}\link\root{}\rootdown{1}
\link\root{}\link\noroot{}\edyn\quad
B'\simeq \dyn \noroot{}\link\root{1}\link\root{}\rootdown{}
\link\root{1}\link\noroot{}\edyn\quad
B''\simeq \dyn \noroot{}\link\root{2}\link\root{}\rootdown{}
\link\root{}\link\noroot{}\edyn\oplus
\dyn \noroot{}\link\root{}\link\root{}\rootdown{}
\link\root{2}\link\noroot{}\edyn.
\end{gather*}
The products $\h^*\otimes B'$ and $\h^*\otimes B''$ have too many components
but
\begin{gather*}
\h^*\otimes B\simeq \dyn \noroot{}\link\root{1}\link\root{}\rootdown{}
\link\root{}\link\noroot{}\edyn \oplus
\dyn \noroot{}\link\root{}\link\root{}\rootdown{}
\link\root{1}\link\noroot{}\edyn\oplus
\dyn \noroot{}\link\root{1}\link\root{}\rootdown{1}
\link\root{}\link\noroot{}\edyn \oplus
\dyn \noroot{}\link\root{}\link\root{}\rootdown{1}
\link\root{1}\link\noroot{}\edyn
\end{gather*}
and so the ALC is satisfied.
\begin{case}[$E^{2,h}_6$] This is a $|2|$-graded geometry for the quasi-split
Lie algebra $E_{6(2)}$.  The horizontal distribution $\cH$ carries the
structure of the spinor representation $S$ of $\so(3,5)$, while the eligible
metrics are induced by the defining representation of $\so(3,5)$.
\end{case}

For $E_7$ and its real forms, irreducibility of $\h$ implies that only one
node may be crossed, and a similar analysis to the $E_6$ type shows that the
cases with the best chance to satisfy the ALC are those with cross over the
first or last node, where
\begin{gather*}
\h^* \simeq \dyn \noroot{}\link\root{1}\link\root{}\link\root{}\rootdown{}
\link\root{}\link\root{}\edyn
\qquad \h \simeq \dyn \noroot{}\link\root{}\link\root{}\link\root{}\rootdown{}
\link\root{}\link\root{1}\edyn\\
S^2\h \simeq \dyn \noroot{}\link\root{1}\link\root{}\link\root{}\rootdown{}
\link\root{}\link\root{}\edyn 
\oplus \dyn \noroot{}\link\root{}\link\root{}\link\root{}\rootdown{}
\link\root{}\link\root{2}\edyn\\
\text{or}\qquad
\h^* \simeq \h\simeq\dyn \root{}\link\root{}\link\root{}\link\root{}\rootdown{}
\link\root{1}\link\noroot{}\edyn\\
S^2\h \simeq \dyn \root{}\link\root{}\link\root{}\link\root{}\rootdown{}
\link\root{2}\link\noroot{}\edyn
\oplus 
\dyn \root{}\link\root{1}\link\root{}\link\root{}\rootdown{}
\link\root{}\link\noroot{}\edyn.
\end{gather*}
It is easy to see that none of these cases satisfy the ALC.

Similarly, for $E_8$, even the most promising candidates
\begin{gather*}
\h^* \simeq \h\simeq\dyn \noroot{}\link\root{1}\link\root{}\link\root{}\link
\root{}\rootdown{}\link\root{}\link\root{}\edyn\\
S^2\h \simeq \dyn \noroot{}\link\root{2}\link\root{}\link\root{}\link\root{}
\rootdown{}\link\root{}\link\root{}\edyn \oplus
\dyn \noroot{}\link\root{}\link\root{}\link\root{}\link\root{}\rootdown{}
\link\root{}\link\root{1}\edyn\\
\text{and}\qquad
\h^*\simeq\dyn\root{}\link\root{}\link\root{}\link\root{}\link\root{}
\rootdown{}\link\root{1}\link\noroot{}\edyn\qquad
\h\simeq\dyn\root{}\link \root{}\link\root{}\link\root{}\link\root{}
\rootdown{1}\link\root{}\link\noroot{}\edyn\\
S^2\h \simeq \dyn \root{}\link\root{}\link\root{}\link\root{}\link\root{}
\rootdown{2}\link\root{}\link\noroot{}\edyn
\oplus\dyn\root{}\link\root{}\link\root{1}\link\root{}\link\root{}\rootdown{}
\link\root{}\link\noroot{}\edyn
\end{gather*}
fail the ALC. Again there can be no cases with more than one cross.

\smallbreak
For $F_4$, the only non-split possibility is
\begin{gather*}
\h^*\simeq \h\simeq
\dyn \noroot{}\link\root{1}\llink<\root{}\link\root{}\edyn
\qquad S^2\h = B\oplus B' \quad\text{where}\quad
B\simeq  \dyn \noroot{}\link\root{2}\llink<\root{}\link\root{}\edyn
\end{gather*}
and $B'$ is trivial. However, $B$ does not satisfy the ALC.

For the split form, all cases can have only one crossed node. When
\begin{gather*}
\h^* \simeq
\dyn \root{1}\link\noroot{}\llink<\root{1}\link\root{}\edyn\qquad
\h\simeq\dyn \root{1}\link\noroot{}\llink<\root{}\link\root{1}\edyn\\
S^2\h= B\oplus B'\qquad
B \simeq \dyn \root{}\link\noroot{}\llink<\root{1}\link\root{}\edyn\qquad
B'\simeq \dyn \root{2}\link\noroot{}\llink<\root{}\link\root{2}\edyn,
\end{gather*}
the elements of $B$ are all degenerate, whereas $\h^*\otimes B'$ does not
satisfy the ALC.

In the remaining two possibilities for the crossed node,
\begin{gather*}
\h^* \simeq  
\dyn \root{}\link\root{2}\llink<\noroot{}\link\root{1}\edyn\qquad
\h\simeq\dyn \root{2}\link\root{}\llink<\noroot{}\link\root{1}\edyn\\
 S^2\h \simeq  
\dyn \root{4}\link\root{}\llink<\noroot{}\link\root{2}\edyn\oplus
\dyn \root{}\link\root{2}\llink<\noroot{}\link\root{2}\edyn \oplus
\dyn \root{2}\link\root{1}\llink<\noroot{}\link\root{}\edyn\\
\text{and}\qquad
\h^* \simeq  \h\simeq
\dyn \root{}\link\root{}\llink<\root{1}\link\noroot{}\edyn\qquad
 S^2\h \simeq  
\dyn \root{}\link\root{}\llink<\root{2}\link\noroot{}\edyn\oplus
\dyn \root{2}\link\root{}\llink<\root{}\link\noroot{}\edyn,
\end{gather*}
the ALC fails in all cases.

\smallbreak
Finally, for $G_2$, only the split case is possible, with one crossed node.
\begin{gather*}
\h^* \simeq  \h\simeq
\dyn \root{3}\lllink<\noroot{}\edyn\qquad
S^2\h \simeq  
\dyn \root{6}\lllink<\noroot{}\edyn\oplus
\dyn \root{2}\lllink<\noroot{}\edyn\\
\text{and}\qquad
\h^* \simeq  \h\simeq
\dyn \noroot{}\lllink<\root{1}\edyn\qquad
S^2\h= B \simeq  
\dyn \noroot{}\lllink<\root{2}\edyn 
\end{gather*}
and only the last of these satisfies the ALC, with
\[
\h^*\otimes B\simeq 
\dyn \noroot{}\lllink<\root{3}\edyn \oplus
\dyn \noroot{}\lllink<\root{1}\edyn.
\]

\begin{case}[$G^{1,1}_2$] The real Lie algebra is the split form of $G_2$ and
the geometry is given by Cartan's famous $(2,3,5)$ distribution. Hence the
horizontal distribution has rank $2$ and the $P$-module $B$ corresponding to
the eligible metrics is the second symmetric power of the defining
representations of $\sgl(2,\R)$.
\end{case}

\section{Examples of reducible cases}\label{s:reducible}

We now discuss a few cases of geometries with reducible $\cH$, where
the linearized metrizability procedure works.  Actually, we have seen several
such examples already, when dealing with real forms with irreducible, but not
absolutely irreducible $\h$ in Theorem~\ref{main}. We list some of
those with irreducible $B$ in the following result.

\begin{thm}\label{more}
The following real parabolic geometries with the Lie algebra $\g$ and choice
of $B$ satisfy the ALC and the linearized metrizability procedure works.
\begin{numlist}  
\item
$B\simeq \dyn \noroot{}\link\root{2}\link\noroot{}\edyn$,
$\g\simeq \sgl(4,\R)$. 
These are Lagrangian contact structures in dimension $5$, where a 
decomposition $\cH= E \oplus F$ of the contact subbundle 
into a direct sum of two Lagrangian subbundles is given. The metrics in $B$
are the split signature metrics with both $E$ and $F$ isotropic.
\item $B\simeq\dyn
\root{1}\dots\root{}\link\noroot{}\link\root{}\dots\root{}\link\noroot{}\link
\root{}\dots\root{1}\edyn$, $\g\simeq \sgl(n+1,\R)$, $n$ even, the first cross
at the $k$-th root $(2k<n)$, crosses at symmetric places. These geometries
come with $\cH$ identified with the sum of two vector bundles of the
form $(E\otimes F^*)\oplus( F\otimes G^*)$, where $E$ and $G$ are real vector
bundles of rank $k$, and $F$ is a real vector bundle of rank $n-2k+1$. The
metrics are the split signature ones, in the subbundle $E\otimes G^*\leq
E\otimes F^*\otimes F\otimes G^*$.
\item $B\simeq\dyn
\root{2}\root{}\link\root{}\dots\root{}\norootupright{}\norootdownright{}\edyn$,
the real Lie algebra is $\so(p,p)$, $2p = n$. The horizontal distribution
$\cH\leq TM$ is the sum of two rank $p-1$ bundles $E$ and $F$ coming
from the defining representations of $\sgl(p-1,\R)$ with different weights,
and $B$ stays for general split metrics on $E\oplus F$.
\item $B\simeq\dyn \noroot{}\link\root{}\link\root{}\rootdown{1}
\link\root{}\link\noroot{}\edyn$, the real Lie algebra is the split form of
type E. The geometry is $|2|$-graded, and the horizontal subspace
$\cH\leq TM$  corresponds to the direct sum of two of the three
isomorphic defining representations of $\so(4,4)$. The eligible metrics are
the generic tracefree split ones and the $P$-module $B$ corresponds
to the third defining representation $\R^{8}$, up to the weight.
\end{numlist}
\end{thm}

\begin{proof}
All cases were already treated for different real forms in the previous
section, except for the very last case. The computation presented there showed
that the ALC is satisfied but the subbundle $\cH$ is not irreducible, but a
sum of two subbundles. At the same time, the strong ALC holds, thus the
linearized metrizability procedure works as required.
\end{proof}

Our final result illustrates the possibility of finding examples with
reducible $B$, including one in which a trivial one-dimensional
component occurs.

\begin{thm}
The following real parabolic geometries with the Lie algebra $\g$ and choice
of $B$ satisfy
the ALC.
\begin{numlist}  
\item $B\simeq\dyn\root{0}\link\noroot{1}\llink<\root{0}\link\noroot{-1}\edyn
  \oplus\dyn\root{0}\link\noroot{-5}\llink<\root{2}\link\noroot{2}\edyn$, the
  real Lie algebra is the split form of type F \textup($|6|$-graded\textup).
  The horizontal distribution $\cH\leq TM$ is built of two rank $2$ bundles
  $E$ and $F$ coming from the defining $\sgl(2,\R)$ representations of the
  different components in $\p_0$.  The first component $\cH_1$ is a tensor
  product $E\otimes F$ with appropriate weight, while $F$ stays for the other
  component $\cH_2$ with another weight. The eligible metrics are the sums of
  the metrics in $\Lambda^2 E\otimes \Lambda^2 F\leq S^2\cH_1$, and the
  metrics in $S^2\cH_2$.
\item 
$B\simeq\dyn\root{}\link\noroot{}\link\root{}\link\root{1}
	\dots\root{}\link\noroot{}\link\root{}\edyn 
	\oplus
	\dyn\root{}\link\noroot{}\link\root{}\dots\root{1}
	\link\root{}\link\noroot{}\link\root{}
	\edyn$.
In this case $\g\simeq \sgl(\rk+1), 5\leq \rk,$ with nodes $2$ and
$\rk-1$ crossed, and $\cH\cong E\otimes F^*\oplus F\otimes G$, where
$E$ is a real vector bundle of rank $2$, and $F$ is a real vector bundle of
rank $\rk-3$. The corank of $\cH\leq TM$ is $4$. The eligible
metrics are sums of the symmetric bilinear forms on $\cH_1$ and
$\cH_2$, both of the form of tensor product of two exterior forms.
\item
$B\simeq \dyn\root{}\link\root{1}\dots\root{}\link\noroot{}\link\root{}
\link\noroot{}\link\root{}\dots\root{}\edyn \oplus
\dyn\root{}\dots\root{}\link\noroot{}\link\root{}
\link\noroot{}\link\root{}\dots\root{1}\link\root{}\edyn$.
Similarly to the previous case, $\g\simeq \sgl(2k),~4\leq k,$ with
nodes crossed at symmetric positions, and $\cH\cong E\otimes F^*\oplus
F\otimes G$, where $E$ and $G$ are real vector bundles of rank $k-1$, while
$F$ is a real vector bundle of rank $2$.  The corank of $\cH\leq TM$
is $(k-1)^2$. The eligible metrics are sums of the symmetric bilinear forms on
$\cH_1$ and $\cH_2$, both of the form of tensor product of two
exterior forms.
\item
$B\simeq\dyn\root{2}\dots\root{}\link\noroot{}\link\noroot{}\link\root{}
\dots\root{}\edyn\oplus
\dyn\root{}\dots\root{}\link\noroot{}\link\noroot{}\link\root{}
\dots\root{2}\edyn$. Here $\g=\sgl(2k+1)$, 
the horizontal distribution is the sum of two
vector bundles of the same rank $k$, corresponding to the defining 
representations of the two semisimple components in $\p_0$. The
metrics are sums of metrics on these two parts of $\cH$.
\end{numlist}
\end{thm}

\begin{proof}
(i) Since the strong ALC cannot hold in the case of split forms of the
  algebras, we work with the complete weights. The form of $\h$ is seen from
  the Cartan matrix of type $F$, while the sum and difference of the second
  and last lines in the inverse Cartan matrix (which corresponds to the
  crossed nodes in the Dynkin diagram) provide the coefficients $(4~8~11~6)$
  and $(2~4~5~2)$ expressing two generating weights in the centre of
  $\p_0$. With their help, we find
\begin{gather*}
\h^* =
\dyn \root{1}\link\noroot{-2}\llink<\root{1}\link\noroot{0}\edyn,
\oplus 
\dyn \root{0}\link\noroot{0}\llink<\root{1}\link\noroot{-2}\edyn.
\\
\h =
\dyn \root{1}\link\noroot{-1}\llink<\root{1}\link\noroot{-1}\edyn
\oplus 
\dyn \root{0}\link\noroot{-2}\llink<\root{1}\link\noroot{1}\edyn.
\end{gather*}
The part of interest in $S^2\h$ is
\[
B_1\oplus B_2 = 
\dyn \root{0}\link\noroot{1}\llink<\root{0}\link\noroot{-1}\edyn \oplus
\dyn \root{0}\link\noroot{-5}\llink<\root{2}\link\noroot{2}\edyn.
\]
Now, $B_1$ is trivial, while 
\begin{gather*}
B_2\otimes\h^*\simeq
\dyn \root{1}\link\noroot{}\llink<\root{3}\link\noroot{}\edyn
\oplus 
\dyn \root{1}\link\noroot{}\llink<\root{1}\link\noroot{}\edyn
\oplus 
\dyn \root{}\link\noroot{}\llink<\root{3}\link\noroot{}\edyn
\oplus 
\dyn \root{}\link\noroot{}\llink<\root{1}\link\noroot{}\edyn.
\end{gather*}
Hence the kernel of $b$ does not exceed the allowed number of 
components and the ALC holds. Finally,
\[
\Lambda^4\h_1 = 
\dyn \root{0}\link\noroot{2}\llink<\root{0}\link\noroot{-2}\edyn
\qquad 
\Lambda^2\h_2=\dyn
\root{0}\link\noroot{-2}\llink<\root{0}\link\noroot{-3}\edyn
\]
so that the weight of $\cL$ can be expressed in terms of them and thus the
linearized metrizability procedure can be completed.

(ii)--(iv) All the other cases have been already discussed as the complex
versions of some cases in the previous section. The only remaining bit of the
proof is the check that the top exterior forms on the individual components
provide linearly independent weights and thus may be used to rescale the
metrics properly. This can be done exactly as in case (i).
\end{proof}

\begin{rem}
Actually, the arguments in the cases (iii) and (iv) above work also in any of
the situations where the crosses are either apart by one or next to each
other, i.e., without assuming they are placed symmetrically, except if the
adjacent crosses appear right at the ends of the diagram. In the latter case
of the so called paths geometries, one of the top degree forms on $\h_i$ has
trivial weight zero and thus the linearized metrizability procedure fails at
the stage when we change the weight of the solutions in order to get genuine
metrics.
\end{rem}

\end{document}